\documentclass[12pt,oneside]{amsart}

\newif\ifpdfstoll
\ifx\pdfoutput\undefined
  \pdfstollfalse
\else
  \pdfoutput=1
  \pdfstolltrue
\fi

\ifpdfstoll
  \usepackage[pdftex]{epsfig}
  \usepackage[pdftex,colorlinks=true,%
              pdftitle={The Mordell-Weil Sieve},%
              pdfauthor={Nils Bruin, Michael Stoll}]{hyperref}
  \newcommand{\Gr}[2]{\psfig{file=#1.pdf,width=#2}}
\else
  \usepackage{epsfig}
  \newcommand{\Gr}[2]{\psfig{file=#1.eps,width=#2}}
\fi

\usepackage{amssymb}  
\usepackage{latexsym} 
\usepackage{comment}
\usepackage[all]{xy}

\DeclareFontEncoding{OT2}{}{} 

\newcommand{\Z}{{\mathbb Z}}
\newcommand{\Q}{{\mathbb Q}}
\newcommand{\R}{{\mathbb R}}

\newcommand{\F}{{\mathbb F}}
\newcommand{\BP}{{\mathbb P}}
\newcommand{\BA}{{\mathbb A}}

\newcommand{\CK}{{\mathcal K}}
\newcommand{\CO}{{\mathcal O}}
\newcommand{\To}{\longrightarrow}

\newcommand{\GL}{\operatorname{GL}}
\newcommand{\Pic}{\operatorname{Pic}}

\newcommand{\Bl}{\operatorname{Bl}}

\newcommand{\tors}{{\text{\rm tors}}}

\newcommand{\eps}{\varepsilon}
\newcommand{\pr}{\operatorname{pr}}
\newcommand{\im}{\operatorname{im}}
\newcommand{\Res}{\operatorname{Res}}
\newcommand{\Spec}{\operatorname{Spec}}
\newcommand{\Char}{\operatorname{char}}
\newcommand{\Grass}{\operatorname{Gr}}
\newcommand{\Prob}{\operatorname{Pr}}

                      {\hspace*{\fill}\nobreak$\Box$\par}

\newenvironment{steplist}[1]%
   {\begin{list}{}{\settowidth{\labelwidth}{#1}%
                   \setlength{\leftmargin}{\labelwidth}%
                   \addtolength{\leftmargin}{\labelsep}%
                   \addtolength{\itemsep}{1mm}}}%
   {\end{list}}

\newtheorem{Theorem}{Theorem}[section]
\newtheorem{Lemma}[Theorem]{Lemma}
\newtheorem{Proposition}[Theorem]{Proposition}
\newtheorem{Corollary}[Theorem]{Corollary}
\newtheorem{Conjecture}[Theorem]{Conjecture}

\theoremstyle{definition}
\newtheorem{Definition}[Theorem]{Definition}

\newtheorem{Remark}[Theorem]{Remark}

\newtheorem{Problem}[Theorem]{Problem}

\numberwithin{equation}{section}

\newcounter{nootje}
\renewcommand\check[1]
  {\marginpar{\tiny\begin{minipage}{20mm}\begin{flushleft}\thenootje : #1\end{flushleft}\end{minipage}}\addtocounter{nootje}{1}}
\begin{document}

\title[The Mordell-Weil sieve]%
      {The Mordell-Weil sieve: Proving non-existence\\ of rational points on 
  curves}

\author{Nils Bruin}
\address{Department of Mathematics,
         Simon Fraser University,
         Burnaby, BC,
         Canada V5A 1S6}
\email{nbruin@cecm.sfu.ca}
\thanks{Research of the first author supported by NSERC}

\author{Michael Stoll}
\address{Mathematisches Institut,
         Universit\"at Bayreuth,
         95440 Bayreuth, Germany.}
\email{Michael.Stoll@uni-bayreuth.de}
\date{November 11, 2009}


\begin{abstract}
  We discuss the Mordell-Weil sieve as a general technique for proving
  results concerning rational points on a given curve. In the special case of
  curves of genus~2, we describe quite explicitly how the relevant local
  information can be obtained if one does not want to restrict to mod~$p$
  information at primes of good reduction. We describe our implementation
  of the Mordell-Weil sieve algorithm and discuss its efficiency.
\end{abstract}

\maketitle


\section{Introduction}

The Mordell-Weil Sieve uses knowledge about the Mordell-Weil group of the
Jacobian variety of a curve, together with local information 
(obtained by reduction mod~$p$, say, for many primes~$p$), in order to
obtain strong results on the rational points on the curve.

The most obvious application that also provided the original motivation for this
work is the possibility to verify that a given curve does not have any
rational points. This is done by deriving a contradiction from the various
bits of local information, using the global constraint that a rational
point on the curve maps into the Mordell-Weil group. This idea is simple
enough (see Section~\ref{S:idea}), but its implementation in form of an 
algorithm that runs in reasonable time on a computer is not completely 
straightforward. The relevant algorithms are discussed in Section~\ref{S:Algo},
and our concrete implementation is described in Section~\ref{S:Imp}.
Section~\ref{S:Eff} contains a discussion of the efficiency of the
implementation and gives some timings.

The idea of using this kind of `Mordell-Weil sieve' computation to prove
that a given curve does not have rational points appears for the first
time in Scharaschkin's thesis~\cite{Scharaschkin}, who used it in a few
examples involving twists of the Fermat quartic. It was then taken up
by Flynn~\cite{Flynn} in a more systematic study of genus~2 curves; his
selection of examples was somewhat biased, however (in favor of curves
he was able to compute with). In our `small curves' project~\cite{BruinStollExp}
we applied the procedure systematically and
successfully to all genus~2 curves $y^2 = f_6 x^6 + \dots + f_1 x + f_0$
with $f_i \in \{-3, -2, -1, 0, 1, 2, 3\}$ that do not possess rational points.

In this situation, it is not strictly necessary to know a full generating
set of the Mordell-Weil group. It is sufficient to know generators of
a finite-index subgroup such that the index is coprime to a certain set
of primes. This can be checked again by using only local information.
In fact the necessary information usually is part of the input for the
sieve procedure. This remark is relevant, since one needs to be able to
compute canonical heights and to enumerate points on the Jacobian up to
a given bound for the canonical height if one wants to obtain generators
for the full Mordell-Weil group. The necessary algorithms are currently
only available for curves of genus~2, see~\cite{StollH1, StollH2}.
We can still use the Mordell-Weil sieve to show that there are no rational
points on a given curve, even when the genus is $\ge 3$. Of course, we
still need to know the Mordell-Weil rank and the right number of independent
points. See~\cite{PSS} for an example where this is applied with a curve
of genus~3 to show that there are no rational points satisfying certain
congruence conditions.

The approach can be modified so that it can be used to verify that there
are no rational points satisfying a given set of congruence conditions
or mapping into a certain coset in the Mordell-Weil group. This is what
was used in~\cite{PSS}. If we can show in addition in some way that in
each of the cosets or residue classes considered, there can be at most
one rational point, then this provides a way of determining the set of
rational points on the curve. Namely, if a given coset or residue class
contains a rational point, then we will eventually find it, and we then
also know that there are no other rational points in this coset or class.
And if there is no rational point in this coset or residue class, then
we can hope to verify this by an application of the Mordell-Weil Sieve.
In this situation, the remark we made above that it is sufficient to know
a finite-index subgroup still applies.

There is one case where we can actually prove that, for a suitable choice
of prime~$p$, no residue class mod~$p$ on the curve can contain more than
one rational point. This is the `Chabauty situation', when the Mordell-Weil
rank is less than the genus. We can (hope to) find a suitable~$p$, and then
we can (hope to) determine the rational points on our curve as outlined
above. This yields a procedure whose termination is not (yet) guaranteed,
since it relies on some conjectures. However, the procedure is correct: if
it terminates, and it has done so in all examples we tried, then it gives
the exact set of rational points on the curve. In the Chabauty context,
the sieving idea has already been used in~\cite{BruinElkies} to rule out
rational points in certain cosets. See also~\cite{PSS} for some more
examples and~\cite{BruinCompMath} for an example that uses `deep' information.

Even when the rank is too large to apply the idea we just mentioned, the
sieve can still be used in order to show that any rational point on the
curve that we have not found so far must be astronomically huge.
This provides at least some kind of moral certainty that there are no other
points. In conjunction with (equally huge) explicit bounds for the size
of {\em integral} points, this allows us to show that we know at least
all the integral points on our curve, see~\cite{IntegralPoints}. For this 
application, however, we really need to know the full Mordell-Weil group,
so with current technology, this is restricted to curves of genus~2.

We discuss these various applications in some detail in Section~\ref{S:Appl}.

In Sections \ref{S:bad} and~\ref{S:deep}, we discuss
how to extract local information that can be used for the sieve, when
we do not want to restrict ourselves to just information mod~$p$ for
primes $p$ of good reduction. In these sections, we assume that the curve
is of genus~2 and that we are working over~$\Q$.

As to the theoretical background, we remark here that under a mild
finiteness assumption on the Shafarevich-Tate group of the curve's
Jacobian variety, the information that can be obtained via the Mordell-Weil
sieve is equivalent to the Brauer-Manin obstruction, see~\cite{Scharaschkin}
or~\cite{StollCov}.

\subsection*{Acknowledgments}

We would like to thank Victor Flynn and Bjorn Poonen
for useful discussions related to our project. Further thanks go
to the anonymous referee for some helpful remarks.
For the computations, the {\sf MAGMA}~\cite{Magma} system was used.


\section{The idea} \label{S:idea}

Let $C/\Q$ be a smooth projective curve of genus~$g \ge 2$
with Jacobian variety~$J$. (In~\cite{StollOw07,StollClay}, we consider more
generally a subvariety of an abelian variety. The idea is the same,
however.)

Our goal is to show that a given curve $C/\Q$ does not have rational
points. For this, we consider the following commuting diagram, where
$v$ runs through the (finite and infinite) places of~$\Q$.
\[ \xymatrix{ C(\Q) \ar[r]^{\iota} \ar[d] & J(\Q) \ar[d]^-{\alpha} \\
              \prod\limits_v C(\Q_v) \ar@<4pt>[r]^{\iota}
                 & \prod\limits_v J(\Q_v)
            }
\]
We assume that we know an embedding $\iota : C \to J$ defined over~$\Q$
(i.e., we know a $\Q$-rational divisor class of degree~$1$ on~$C$) and
that we know generators of the Mordell-Weil group~$J(\Q)$. If $C(\Q)$ is
empty, then the images of~$\alpha$ and the lower $\iota$ are
disjoint, and conversely.

However, since the sets and groups involved are infinite, 
we are not able to compute this intersection. Therefore, we replace
the groups by finite approximations. Let $S$ be a finite set of places of~$\Q$
and let $N \ge 1$ be an integer. Then we consider
\[ \xymatrix{ C(\Q) \ar[r]^-{\iota} \ar[d] & J(\Q)/N J(\Q) \ar[d]^-{\alpha} \\
              \prod\limits_{v \in S} C(\Q_v) \ar@<4pt>[r]^-{\beta}
                 & \prod\limits_{v \in S} J(\Q_v)/N J(\Q_v)
            }
\]
Under the assumptions made, we now {\em can} compute the images
of~$\alpha$ and of~$\beta$ and check if they are disjoint.
If $C(\Q) = \emptyset$, then according
to the Main Conjecture of~\cite{StollCov} and the heuristic given
in~\cite{PoonenHeur}, the two images should be disjoint 
when $S$ and $N$ are large enough. Note that (as shown in~\cite{StollCov})
the two images will be disjoint for some choice of $S$ and~$N$ if and
only if $\prod_v \iota(C(\Q_v))$ does not meet the topological closure
of~$J(\Q)$ in $\prod_p J(\Q_p) \times J(\R)/J(\R)^0$, where $J(\R)^0$
denotes the connected component of the origin. This is a stronger
condition than the requirement that $\prod_v \iota(C(\Q_v))$ misses
the image of~$J(\Q)$. The conjecture claims that both statements are
in fact equivalent.

As a further simplification, we can just use a set $S$ of primes of
good reduction and replace the above diagram by the following simpler one:
\[ \xymatrix{ C(\Q) \ar[r]^-{\iota} \ar[d] & J(\Q)/N J(\Q) \ar[d]^-{\alpha} \\
              \prod\limits_{p \in S} C(\F_p) \ar@<4pt>[r]^-{\beta}
                 & \prod\limits_{p \in S} J(\F_p)/N J(\F_p)
            }
\]
Poonen originally formulated his heuristic for this case. However, in practice
it appears to be worthwhile to also use `bad' information
(coming from primes of bad reduction) and `deep' information (involving
parts of the kernel of reduction) in order to keep the running time of
the actual sieve computation within reasonable limits.
In Sections \ref{S:bad} and~\ref{S:deep} below, we show how to obtain
this kind of information for curves of genus~2 over~$\Q$.


\section{Algorithms} \label{S:Algo}

In the following, we assume that we are using the simpler version involving
only reduction mod~$p$, as described at the end of Section~\ref{S:idea}.

Let $r$ denote the rank of the Mordell-Weil group~$J(\Q)$. For a given
set~$S$ and parameter~$N$, denote by $A(S, N) \subset J(\Q)/N J(\Q)$
the subset of elements mapping
into the image of~$C(\F_p)$ in $J(\F_p)/N J(\F_p)$ for all $p \in S$,
in symbols:
\[ A(S, N) = \{a \in J(\Q)/N J(\Q) : \alpha(a) \in \im(\beta_{N,p})
                                     \text{\ for all $p \in S$}\}
\]
Here, $\beta_{N,p} : C(\F_p) \to J(\F_p)/N J(\F_p)$ denotes the composition
of $\iota : C(\F_p) \to J(\F_p)$ and the canonical epimorphism
$J(\F_p) \to J(\F_p)/N J(\F_p)$.

The procedure splits into three parts. 
\begin{enumerate}
\setlength{\parskip}{0.8ex plus 0.1ex minus 0.1ex}
\setlength{\parindent}{0mm}

  \item {\bf Choice of~$S$}

    In the first step, we have to choose a set~$S$
    of primes such that we can be reasonably certain that the combined information
    obtained from reduction mod~$p$ for all $p \in S$ is sufficient to
    give a contradiction (or, more generally, to have $A(S, N)$ equal to
    the image of~$C(\Q)$, for suitable~$N$). In Section~\ref{SubS:S}, we
    explain a criterion that tells us if $S$~is likely to be good for our
    purposes. The actual computation of the relevant local information
    is also part of this step. For each prime $p \in S$, we find the
    abstract finite abelian group~$G'_p$ representing~$J(\F_p)$ (or some
    other finite quotient of~$J(\Q_p)$)
    and the image $X'_p \subset G'_p$ of $\iota : C(\F_p) \to J(\F_p)$.
    We also compute the homomorphism $\phi_p : J(\Q) \to G'_p$. We write
    $G_p$ for the image of~$\phi_p$ and denote $X_p = X'_p \cap G_p$.
    
    In what follows below, we will use $\#X_p/\#G_p$ as a measure for how
    much information about rational points on~$C$ can be obtained at~$p$.
    Note that it is possible that $G_p \subset X'_p \subsetneq G'_p$.
    In that case, $\#X'_p/\#G'_p < 1$, but no element of the Mordell-Weil
    group can be ruled out from coming from~$C(\Q)$, based on the information
    at~$p$. If we were to use this quantity, we would obtain
    erroneous estimates in the second step. This can then lead to huge
    sets $A(S, N)$ in the third step and even to a failure of the computation.

  \item {\bf Choice of~$N$}

    In the second step, we fix a target value of~$N$ and determine a way
    to compute $A(S, N)$ efficiently. We do that by finding an ordered
    factorization $N = q_1 q_2 \cdots q_m$ such that none of the intermediate
    sets $A(S, q_1 \cdots q_k)$ becomes too large.
    This is explained in Section~\ref{SubS:N}.

  \item {\bf Computation of $A(S, N)$}

    Finally, we have to actually compute~$A(S, N)$ in a reasonably efficient
    way. We explain in Section~\ref{SubS:A} how this can be done.
\end{enumerate}

The last two steps can be considered independently from the Mordell-Weil
sieve context. Basically, we need a procedure that, given a finite family
of surjective group homomorphisms
$\phi_i : \Gamma \to G_i$ and subsets $X_i \subset G_i$,
(for $i \in I$) attempts to prove that for every $a \in \Gamma$
there is some $i \in I$ such that $\phi_i(a) \notin X_i$.
Here $\Gamma$ is a finitely generated abelian group and the $G_i$
are finite abelian groups. In our application,
$\Gamma$ is the Mordell-Weil group, the index set is~$S$, $G_p$ is the
image of $J(\Q)$ in~$J(\F_p)$,
and $X_p = \iota\bigl(C(\F_p)\bigr) \cap G_p$. 

We give some more details on our actual implementation in Section~\ref{S:Imp}.


\subsection{Choice of~$S$} \label{SubS:S}

The first task of the algorithm is to come up with a suitable set~$S$
of places. We will restrict to finite places (i.e., primes), but in
principle, one could also include information at infinity, which would
mean to consider the connected components of~$J(\R)$ which meet the
image of~$C(\R)$ under the embedding~$\iota$.

It is clear that the only possibility to get some interaction between the
information at various primes~$p$ (and eventually a contradiction) is 
when the various group orders $\#J(\F_p)$ have common factors. This is
certainly more likely when these common factors are relatively small.
We therefore look for primes~$p$ (of good reduction) such that the group
order $\#J(\F_p)$ is $B$-smooth (i.e., with all prime divisors $\le B$)
for some fixed value of~$B$; in practice,
values like $B = 100$ or $B = 200$ lead to good results.

For each such prime, we compute the group structure of $J(\F_p)$, i.e.,
an abstract finite abelian group $G'_p$ together with an explicit isomorphism
$J(\F_p) \cong G'_p$. We also compute the images of the generators of~$J(\Q)$
in~$G'_p$ and the image of $C(\F_p)$ in~$G'_p$. In order to do that, we need to solve 
roughly $p$ discrete logarithm problems in $G'_p$. Since $G'_p$ has smooth 
order, we can use Pohlig-Hellman reduction~\cite{PohligHellman} to reduce to a number 
of small discrete log problems. Therefore, this part of the computation is 
essentially linear in $p$ in practice.
We do need to compute reasonably efficiently in $J(\F_p)$, though.
If $C$ is a curve of genus~$2$, Cantor reduction \cite{Cantor} gives us a way
to do that. To fix notation, let $W$ denote an
effective canonical divisor on~$C$. Cantor reduction takes as input a degree~$0$
divisor in the form $D - d W$, where $D$ is an effective divisor of degree~$2d$,
and computes a unique divisor $D_0$ of degree~$2$ such that
\[ [D - d W] = [D_0 - W]\,, \]
with the convention that if $D - d W$ is principal, then $D_0 = W$.
Adding two divisor classes represented as $[D_1 - W]$ and 
$[D_2 - W]$ can be accomplished by feeding the divisor 
$(D_1+D_2) - 2 W$ into the reduction algorithm.

Cantor reduction also allows us to map elements from $C(\F_p)$ into~$J(\F_p)$.
If $\iota$ is given by a rational base point $P_0 \in C(\Q)$, i.e., 
$\iota(P) = [P - P_0]$, and $\bar{P}_0$ is the reduction of~$P_0$ modulo~$p$, 
then for each $\bar{P} \in C(\F_p)$, we have 
$\iota(\bar{P}) = [\bar{P} + \bar{P}'_0 - \bar{W}]$, 
where $\bar{P}'_0$ is the hyperelliptic involute of $\bar{P}_0$. 
In this case, we already get $\iota(\bar{P})$ as a reduced divisor class.
Otherwise, $\iota$ is given by $\iota(P) = [P - D_3 + W]$, where $D_3$ 
is a rational effective divisor of degree~$3$. Then we can compute a reduced 
representative of $\iota(\bar{P})$ by performing Cantor reduction on 
$(\bar{P} + \bar{D}'_3) - 2W$.

As mentioned above, we finally replace $G'_p$ by $G_p = \phi_p\bigl(J(\Q)\bigr)$,
and we let $C_p$ be the intersection of the image of~$C(\F_p)$ in~$G'_p$
with~$G_p$. We then use $\phi_p$ to denote the surjective homomorphism
$\phi_p : J(\Q) \to G_p$.

In order to determine whether we have collected enough primes, we compute
the expected size of the set~$A(S, N)$, where $S$ is the set of all $p$ 
collected so far and $N$ is a suitable value as specified below.
We follow Poonen~\cite{PoonenHeur} and assume that the images
of the $C(\F_p)$ in~$J(\F_p)$ are random and independent for the various~$p$.
This leads to the expected value
\[ n(S, N) = \#\bigl(J(\Q)/N J(\Q)\bigr)
            \prod_{p \in S} \frac{\# C_{N,p}}{\#\bigl(G_p/N G_p\bigr)}
\]
where $C_{N,p}$ is the image of $C_p$ in $G_p/N G_p$.

In principle, we would like to find the value of~$N$ that minimizes~$n(S, N)$
for the given set~$S$. However, this would lead to much too involved a
computation. We therefore propose to proceed as follows. Write
\[ \prod_{p \in S} J(\F_p)
     \cong \Z/N_1\Z \times \Z/N_2\Z \times \cdots \times \Z/N_l\Z
\]
where $N_j$ divides~$N_{j+1}$ ($j = 1, 2, \dots, l-1$). Then we take
$N = N_{l-r-1-j}$ for $j = 0, 1, 2, 3$ as values that are likely
to produce a small~$n(S, N)$. The reason for this choice is the following.
Usually the target groups will be essentially cyclic, and the kernel
of the homomorphism $J(\Q) \to J(\F_p)$ will be a random subgroup
of index~$\#J(\F_p)$ and more or less cyclic quotient. If we take
a prime number~$q$ for~$N$ and the Mordell-Weil rank is~$r$, then we
obtain a random codimension one subspace of $\F_q^r$. Unless $q$ is
very small, it will be rather unlikely that these subspaces intersect
in a nontrivial way, unless there are more than~$r$ of them. So for
every prime power dividing our~$N$, we want to have more than~$r$
factors in the product above that have order divisible by the prime power.
So we should restrict to divisors of~$N_{l-r-1}$. Taking $N_{l-r-1-j}$
with $j > 0$, we make sure to get even more independent factors.

By the same token, any subgroup $L \subset J(\Q)$ such that we can
expect to get sufficient information on the image of~$C(\Q)$
in~$J(\Q)/L$ will be very close to $N J(\Q)$ for some~$N$: as soon
as the various bits of information interact, we will have exhausted
all ``directions'' in the dual of $J(\Q)/N J(\Q)$, and the intersection
of the kernels of the relevant maps will be close to~$N J(\Q)$.
This also explains why our approach to the computation of~$A(S, N)$,
which we describe below in Section~\ref{SubS:A}, works quite well.

Note that by taking $S$ (and perhaps also~$B$) large, we will get large
values for the number~$l$ of factors. Once $l \gg r$, the image of the
Mordell-Weil group $J(\Q)$ in this product will be rather small, so that
we can expect it to eventually miss the image of the curve. Poonen's
heuristic~\cite{PoonenHeur} makes this argument precise.

We continue collecting primes into~$S$ until we find a sufficiently
small~$n(S, N)$. In practice, it appears that $n(S, N) < \eps = 10^{-2}$ is
sufficient. Note that if the final sieve computation is unsuccessful 
(and does not
lead to the discovery of a rational point on~$C$), then we can enlarge~$S$
until $n(S, N)$ gets sufficiently smaller and repeat the sieve computation.


\subsection{Choice of~$N$} \label{SubS:N}

Once $S$ is chosen and the relevant information is computed, we can
forget about the original context and consider the following more
abstract situation.

We are given a finitely generated abstract abelian group~$\Gamma$ of rank~$r$,
together with a finite family $(G_i, \phi_i, X_i)_{i \in I}$ of triples,
where $G_i$ is a finite abstract abelian group, $\phi_i : \Gamma \to G_i$
is a surjective homomorphism, and $X_i \subset G_i$ is a subset. In practice, $\Gamma$
and the~$G_i$ are given as a product of cyclic groups, $\phi_i$ is given
by the images of the generators of~$\Gamma$, and $X_i$ is given by enumerating
its elements. The following definition generalizes $A(S, N)$.

\begin{Definition}
  Let $L \subset \Gamma$ be a subgroup of finite index. We set
  $G_{L,i} = G_i/\phi_i(L)$, write $X_{L,i}$ 
  for the image of~$X_i$ in~$G_{L,i}$, and denote by~$\phi_{L,i}$ the induced
  homomorphism $\Gamma/L \to G_{L,i}$. We define
  \[ A(L) = \{\gamma \in \Gamma/L
               : \phi_{L,i}(\gamma) \in X_{L,i} \text{\ for all $i \in I$}\}
  \]
  and its expected size
  \[ n(L) = \#\bigl(\Gamma/L\bigr) \prod_{i \in I} \frac{\# X_{L,i}}{\# G_{L,i}} \,.
  \]
\end{Definition}

Now the task is as follows.

\begin{Problem}\strut
  \begin{enumerate}
    \item Find a number~$N$ such that $A(N\Gamma)$ has a good chance of being
          empty and such that $A(N\Gamma)$ can be computed efficiently.
    \item Compute $A(N\Gamma)$.
  \end{enumerate}
\end{Problem}

In our application, $\Gamma = J(\Q)$, $I = S$, and for
$p \in S$, $G_p$ and $\phi_p$ are as before, and $X_p = C_p$.

Since we may have to take $N$ fairly large ($N \approx 10^6$ is not
uncommon, and values $\approx 10^{12}$ or even $\approx 10^{100}$ do occur
in practice in our applications), it would not be a good idea to enumerate
the (roughly $N^r$)
elements of $\Gamma/N \Gamma$ and check for each of them whether it
satisfies the conditions. Instead, we build up~$N$ multiplicatively
in stages: we compute $A(N_j \Gamma)$ successively for a sequence of values
\[ N_0 = 1, \quad N_1 = q_1, \quad N_2 = N_1 q_2, \quad N_3 = N_2 q_3,
   \quad\dots, \quad N_m = N_{m-1} q_m = N
\]
where the $q_k$ are the prime divisors of~$N$. We want to choose the
sequence $(q_k)$ (and therefore~$N$) in such a way that the intermediate
sets $A(N_k \Gamma)$ are likely to be small. For this, we use again the
expected size $n(N_k \Gamma)$ of~$A(N_k \Gamma)$. By a best-first search, we find
the sequence $(q_k)_{k=1,\dots,m}$ such that
\begin{enumerate}
  \item[(i)] $n\bigl((\prod_{k=1}^m q_k) \Gamma\bigr)$
             is less than a target value~$\eps_1 < 1$ (for example, $0.1$), and 
  \item[(ii)] $\max \bigl\{n(N_k \Gamma) : 0 \le k \le m\bigr\}$ is minimal 
              (where $N_k = \prod_{j=1}^k q_j$).
\end{enumerate}

From the first step, which provides the input, we can deduce a number~$M$
(usually $M = N_{l-1-r-j}$ for some small value of~$j$, in the notation
used above) such that all reasonable choices for~$N$ should divide~$M$.
The following procedure returns a suitable sequence $(q_1, \dots, q_m)$.

\strut\hrulefill \\[3pt]
{\sf FindQSequence}: \\
\strut\quad $c$ := $\{\bigl((), 1, 1.0\bigr)\}$ 
 \qquad // () is an empty sequence of~$q_k$, 1 is $N$, 1.0 is $n(N \Gamma)$ \\
\strut\quad {\sf while} $c \neq \emptyset$: \\
\strut\qquad $(s, N, n)$ := triple in~$c$ with minimal~$n$ \\
\strut\qquad remove this triple from~$c$ \\
\strut\qquad {\sf if} $n < \varepsilon$: \qquad // success? \\
\strut\qquad\quad {\sf return} $s$ \\
\strut\qquad {\sf end if} \\
\strut\qquad // compute the possible extensions of~$s$ and add them to the list \\
\strut\qquad $c$ := $c \cup \bigl\{\bigl(\text{append}(s, q), Nq, n(Nq \Gamma)\bigr)
                                    : q \text{\ prime}, Nq \mid M \bigr\}$ \\
\strut\quad {\sf end while} \\
\strut\quad // if we leave the while loop here, the target was not reached \\
\strut\quad {\sf return} `failure' \\[-3pt]
\strut\hrulefill

\smallskip

When we extend~$c$, we can restrict to the triples $(s', N', n')$ such
that $N'$ does not occur as the second component of a triple already in~$c$.
(Since in this case, we have already found a `better' sequence leading
to this~$N'$.)

If the information given by $(G_i, \phi_i, X_i)_{i \in I}$ is sufficient
(as determined in the first step),
then this procedure usually does not take much time (compared to the
computation of the `local information' like the image of~$C(\F_p)$
in~$J(\F_p)$). In any case, if we made sure in the first step
that there is some $M$ such that $n(M \Gamma) < \eps_1$, then
{\sf FindQSequence} will not fail.

In this step and also in the first step, it is a good idea to keep the orders of
the cyclic factors of the groups~$G_i$ and the numbers~$N$ in factored
form, and only convert the greatest common divisors of~$N$ with the relevant
group orders into actual integers.


\subsection{Computation of $A(N \Gamma)$} \label{SubS:A}

Now we have fixed the sequence $(q_k)_{j=1,\dots,m}$ of primes whose
product is~$N$. In the last part of the algorithm, we have to compute
the set~$A(N \Gamma)$ (and hope to find that it is empty or sufficiently
small, depending on the intended application).

This is done iteratively, by successively computing $A(N_k \Gamma)$,
where $N_k = \prod_{j=1}^k q_j$. We start at $k = 0$ and initialize
$A(N_0 \Gamma) = A(\Gamma) = \{0\} \subset \Gamma/\Gamma$. Then, assuming
we know $A(N_{k-1} \Gamma)$, we compute $A(N_k \Gamma)$ as follows.

We first find the triples $(G_i, \phi_i, X_i)$ that can possibly provide
new information.
The relevant condition is that $v_{q_k}(e_i) \ge v_{q_k}(N_k)$, where
$e_i$ is the exponent of the group~$G_i$. For these~$i$, we compute
the group $G_{N_k \Gamma, i}$, the image~$X_{N_k \Gamma, i}$ 
of~$X_i$ in this group and the homomorphism
$\phi_{N_k \Gamma, i} : \Gamma/N_k \Gamma \to G_{N_k \Gamma, i}$.

The most obvious approach now would be to take each $\gamma \in A(N_{k-1} \Gamma)$,
run through its various lifts to~$\Gamma/N_k \Gamma$ and check for each
lift if it is mapped into $X_{N_k \Gamma, i}$ under~$\phi_{N_k \Gamma, i}$.
The complexity of this procedure is $\#A(N_{k-1} \Gamma) \cdot q_k^r$
times the average number of tests we have to make
(we disregard possible torsion in~$\Gamma$, which will not play a role
once $N_{k-1}$ is large enough).
Unless $r$ is very small, the procedure will be rather slow
when the intermediate sets~$A(N_k \Gamma)$ get large.

In order to improve on this, we split the inclusion 
$N_k \Gamma \subset N_{k-1} \Gamma$ into several stages:
\[ N_{k-1} \Gamma = L_0 \supset L_1 \supset \dots \supset L_t = N_k \Gamma \,. \]
Note that the quotient $N_{k-1} \Gamma/N_k \Gamma$ is isomorphic to
$(\Z/q_k \Z)^r$ (again disregarding torsion in~$\Gamma$), so we can hope to get
up to~$r$ intermediate steps. We now proceed as follows.

\strut\hrulefill\\[3pt]
{\sf PrepareLift}($k$): \\
\strut\quad $j$ := 0; $L_0$ := $N_{k-1} \Gamma$ \qquad // initialize \\
\strut\quad $I'$ := $\{i \in I : v_{q_k}(e_i) \ge v_{q_k}(N_k)\}$ 
  \qquad // the relevant subset of~$I$ \\
\strut\quad {\sf while} $I' \neq \emptyset$: \\
\strut\qquad $j$ := $j + 1$ \\
\strut\qquad // list the possible subgroups for the next step \\
\strut\qquad $\Lambda$ := $\bigl\{L_{j-1} \cap \ker(\phi_i) : i \in I'\bigr\}$ \\
\strut\qquad {\sf for} $L \in \Lambda$: \\
\strut\qquad\quad // compute a measure of how `good' each subgroup is \\
\strut\qquad\quad $n(L_{j-1}, L)$ := $\displaystyle (L_{j-1} : L)
                \prod_{i \in I'}
                  \frac{\#X_{L,i}}{\#X_{L_{j-1},i}}
                  \,\frac{1}{\bigl(\phi_i(L_{j-1}) : \phi_i(L)\bigr)}$ \\
\strut\qquad {\sf end for} \\
\strut\qquad $L_j$ := the $L \in \Lambda$ that has the smallest~$n(L_{j-1}, L)$ \\
\strut\qquad // record the $i \in I'$ that contribute to this step \\
\strut\qquad $I_j$ := $\{i \in I' : \phi_i(L_j) \neq \phi_i(L_{j-1})\}$ \\
\strut\qquad $I'$ := $\{i \in I' : \phi_i(L_j) \not\subset N_k G_i\}$
  \qquad // update $I'$ \\
\strut\quad {\sf end while} \\
\strut\quad {\sf if} $L_j \neq N_k \Gamma$: \\
\strut\qquad // fill the remaining gap to $N_k \Gamma$ \\
\strut\qquad $t$ := $j+1$; $L_{t}$ := $N_k \Gamma$; $I_{t}$ := $\emptyset$ \\
\strut\quad {\sf else} \\
\strut\qquad $t$ := $j$ \\
\strut\quad {\sf end if} \\[-3pt]
\strut\hrulefill


The quantity $n(L_{j-1}, L)$ that we compute in the algorithm above is the expected
number of ``offspring'' that an element of $A(L_{j-1})$ generates
in~$A(L)$.

We then successively compute $A(L_1)$, \dots, $A(L_t) = A(N_k \Gamma)$ in
the same way as described above for the one-step procedure:

\strut\hrulefill\\[3pt]
{\sf Lift}($k$): \\
\strut\quad // note that $A(L_0) = A(N_{k-1} \Gamma)$ \\
\strut\quad {\sf for} $j$ = 1, \dots, $t$: \\
\strut\qquad $A(L_j)$ := $\emptyset \subset \Gamma/L_j$ \\
\strut\qquad {\sf for} $a \in A(L_{j-1})$: \\
\strut\qquad\quad $a'$ := a representative of $a$ in $\Gamma/L_j$ \\
\strut\qquad\quad {\sf for} $l \in L_{j-1}/L_j$: \\
\strut\qquad\qquad {\sf if} $\forall i \in I_j : \phi_{L_j,i}(a'+l) \in X_{L_j,i}$: \\
\strut\qquad\qquad\quad $A(L_j)$ := $A(L_j) \cup \{a'+l\}$ \\
\strut\qquad\qquad {\sf end if} \\
\strut\qquad\quad {\sf end for} \\
\strut\qquad {\sf end for} \\
\strut\quad {\sf end for} \\
\strut\quad // now $A(N_k \Gamma) = A(L_t)$ \\
\strut\quad {\sf return} \\[-3pt]
\strut\hrulefill

\smallskip

In practice, {\sf PrepareLift} and~{\sf Lift} together form one
subroutine, whose input is $(N,q,A) = (N_{k-1},q_k,A(N_{k-1} \Gamma))$
(together with the global data $\Gamma$ and $(G_i, \phi_i, X_i)_{i \in I}$)
and whose output is $A(N q \Gamma)$ (with $N q = N_k$).

The complexity of the lifting step is now
\[ \sum_{j=1}^t \#A(L_{j-1}) (L_{j-1} : L_j)
     \approx \#A(N_{k-1} \Gamma)
             \sum_{j=1}^t (L_{j-1} : L_j) \prod_{i=1}^{j-1} n(L_{i-1}, L_i) \,.
\]
In the worst case, we have $n(L_{j-1}, L_j) = (L_{j-1} : L_j)$; then
the second factor is at most 
$q_k + q_k^2 + \dots + q_k^{r} < \frac{q_k}{q_k-1} q_k^r$; this is not
much worse than the factor~$q_k^r$ we had before. Usually, however, 
and in particular when $N_{k-1}$ is already fairly large, the numbers
$n(L_{j-1}, L_j)$ will be much smaller than $(L_{j-1} : L_j)$;
also we should have $t = r$
and $(L_{j-1} : L_j) = q_k$, so that the complexity is essentially
$\#A(N_{k-1} \Gamma) q_k$. As an additional benefit, we distribute the
tests we have to make over the intermediate steps, so that the
average number of tests in the innermost loop will be smaller than
when going directly from $N_{k-1} \Gamma$ to~$N_k \Gamma$.

In this way, it is possible to compute these
sets even when $r$ is not very small. For example, in order to find
the integral solutions of $\binom{y}{2} = \binom{x}{5}$
(see~\cite{IntegralPoints}), it was necessary
to perform this kind of computation for a group of rank~6, and this
was only made possible by our improvement of the lifting step.
As another example, one of the two rank~4 curves that had to be dealt
with by the Mordell-Weil sieve in our experiment~\cite{BruinStollExp}
took the better part of a day with the implementation we had at the
time (which was based on the ``obvious approach'' mentioned above).
With the new method, this computation takes now less than 15~minutes.

\medskip

If we find that $A(N_k \Gamma) = \emptyset$ for some $k \le m$, then we
stop. In the context of our application, this means that we have proved
that $C(\Q) = \emptyset$ as well. Otherwise, 
we can check to see if the remaining elements in~$A(N \Gamma)$
actually come from rational points by computing the element of~$J(\Q)$
of smallest height that is in the corresponding coset. It is usually
a good idea to first do some more mod~$p$ checks so that one can be
certain that the point in~$J(\Q)$ really gives rise to a point in~$C(\Q)$.
If we do not find a rational point on~$C$ in this way, then we can
increase~$S$ and decrease $\eps$ and~$\eps_1$ and repeat the computation.

Let us also remark here that the lifting step can easily be parallelized,
since we can compute the ``offspring'' of the various $a \in A(N_{k-1} \Gamma)$
independently. After the preparatory computation in {\sf PrepareLift}
has been done, we can split $A(N_{k-1} \Gamma)$ into a number of subsets
and give each of them to a separate thread to compute the resulting
part of~$A(N_k \Gamma)$. Then the results are collected, we check if the
new set is empty, and if it is not, we repeat this procedure with the
next lifting step.


\section{Applications} \label{S:Appl}

\subsection{Non-Existence of Rational Points}

The main application we had in mind (and in fact, the motivation for
developing the algorithm described in this paper) is in the context
of our project on deciding the existence of rational points on all
`small' genus~$2$ curves, see the report~\cite{BruinStollExp}.

Out of initially about 200\,000 isomorphism classes of curves, there are
1492 that are undecided after a search for rational points, checking
for local points, and a $2$-descent~\cite{BruinStoll2D}.
We applied our algorithm to these
curves and were able to prove for all of them that they do not have
rational points. For some curves, we
needed to assume the Birch and Swinnerton-Dyer conjecture for the
correctness of the rank of the Mordell-Weil group.

For the curves whose Jacobians have rank at most~$2$, we originally only used
`good' and
`flat' information, i.e., groups $J(\F_p)$ for primes $p$ of good reduction.
For ranks $3$ and~$4$ (no higher ranks occur), we also used `bad'
and `deep' information, as described in Sections \ref{S:bad} 
and~\ref{S:deep} below. The running time of the Magma implementation
of the Mordell-Weil sieve algorithm we had at the time was about one day
for all 1492 curves
(on a 1.7~GHz machine with 512~MB of RAM). Two thirds of that time was
taken by one of the two rank~$4$ curves, and most of the remaining time
was used for the 152 rank~$3$ curves.

With the current implementation discussed in Section~\ref{S:Imp} below,
the overall running time (now on a 2.0~GHz machine with 4~GB of RAM)
is about two and a half hours. For a detailed discussion of the timings,
see Section~\ref{S:Eff}.


\subsection{Finding points}

Instead of proving that no rational points on~$C$ exist, we can also use
the Mordell-Weil sieve idea in order to {\em find} rational points on~$C$
up to very large height. When the rank is less than the genus, we can
even combine the Mordell-Weil sieve with Chabauty's method in order to
compute the set of rational points on~$C$ exactly, see Section~\ref{SubS:Chab}
below.

We want to find the rational points on~$C$ up to a certain (large)
logarithmic height bound~$H$. We assume that we know the height pairing
matrix for the generators of~$J(\Q)$ and a bound for the difference
between naive and canonical height on~$J(\Q)$. See~\cite{StollH1,StollH2}
for algorithms that provide these data in the case of genus~2 curves.
From this information and the embedding $C \to J$, we can then compute
constants $\delta$ and~$d$
such that $\hat{h}(\iota(P)) \le d\,h(P) + \delta$
for all $P \in C(\Q)$. Here $\hat{h}$ denotes the canonical height on~$J(\Q)$
and $h$ denotes a suitable height function on the curve. The upshot of this
is that $h(P) \le H$ implies $\hat{h}(\iota(P)) \le H' = dH + \delta$.

Note that in many cases when we want to find all rational points up to
height~$H$, we already know a rational point $P_0$ on~$C$. Then we can 
just use $P \mapsto [P - P_0]$ for the embedding~$\iota$.

We now proceed as before: we find a suitable
set $S$ of primes and a number~$N$ and compute $A(S, N) \subset J(\Q)/NJ(\Q)$.
For the purposes of this application, we require $N$ to be divisible by
the exponent of the torsion group $J(\Q)_{\tors}$ and to be such that
$N^2 > 4 H'/m$, where $m$ is the minimal canonical height of a non-torsion
point in~$J(\Q)$. These conditions imply that if $Q, Q' \in J(\Q)$ are
such that $Q - Q' \in NJ(\Q)$ and $\hat{h}(Q), \hat{h}(Q') \le H'$, then
$Q = Q'$. In other words, each coset of $NJ(\Q)$ in~$J(\Q)$ contains at
most one point of canonical height~$\le H'$.

We do not necessarily expect $A(S, N)$ to be empty now. However, by the 
preceding discussion, each element of~$A(S, N)$ corresponds to at most
one point in~$C(\Q)$ of height~$\le H$. Therefore we consider the elements
of~$A(S, N)$ in turn (we expect them to be few in number), and for each
of them, we do the following. First we check whether there is an element
$Q$ in the corresponding coset of~$NJ(\Q)$ such that $\hat{h}(Q) \le H'$.
If this is not the case, we discard the element. Otherwise, there is only
one such~$Q$, and we check for some more primes $p \notin S$ whether
the image of~$Q$ in~$J(\F_p)$ is in the image of~$C(\F_p)$. Note that we
can perform these tests quickly only based on the representation of~$Q$
as a linear combination of the generators of~$J(\Q)$: we reduce the
generators mod~$p$ and compute the reduction of~$Q$ as a linear combination
of the reduced generators. Depending
on~$H'$, we can determine such a set of primes beforehand, with the property
that a point~$Q \in J(\Q)$ with $\hat{h}(Q) \le H'$
that `survives' all these tests must be in~$\iota(C(\Q))$, see the lemma below.
So if $Q$ fails one of the tests, we discard it, otherwise
we compute $Q$ as an explicit point and find its preimage in~$C(\Q)$
under~$\iota$.

\begin{Lemma}
  Let $P_0 \in C(\Q)$ and write $x(P_0) = (a : b)$ with coprime integers $a,b$.
  Let $p_1, p_2, \dots, p_m$ be primes of good reduction such that 
  \[ p_1 p_2 \cdots p_m > e^{H' + \gamma} \max\{|a|,|b|\}^2 \]
  and such that $P_0$ and its hyperelliptic conjugate $\bar{P}_0$ are distinct
  mod some~$p_{j_0}$ if they are distinct in~$C(\Q)$.
  Here $\gamma$ is a bound for the difference $h - \hat{h}$ between naive
  and canonical height on~$J(\Q)$. We take $\iota : P \mapsto [P - P_0]$.
  
  If $Q \in J(\Q)$ satisfies $\hat{h}(Q) \le H'$ and is such that
  the reduction of~$Q$ mod~$p_j$
  is in $\iota(C(\F_{p_j}))$ for all $1 \le j \le m$, then $Q \in \iota(C(\Q))$.
\end{Lemma}

\begin{proof}
  Let $(k_1 : k_2 : k_3 : k_4)$ be the image of~$Q$ on the Kummer surface
  of~$J$, with coprime integers~$k_j$. If $Q$ mod~$p_j$ is on the image
  of the curve, then $p_j$ divides $k_1 b^2 - k_2 a b + k_3 a^2$. This
  integer has absolute value at most $e^{H' + \gamma} \max\{|a|,|b|\}^2$,
  so if it is divisible by $p_1, \dots, p_m$, it must be zero. This
  implies that $Q = [P - P_0]$ or $Q = [P - \bar{P}_0]$ for some $P \in C(\Q)$.
  If $P_0 \neq \bar{P}_0$, these two cases can be distinguished mod~$p_{j_0}$.
\end{proof}

The test whether a given coset of~$NJ(\Q)$ contains a point of canonical
height~\hbox{$\le H'$} comes down to a `closest vector' computation with respect
to the lattice $(NJ(\Q), \hat{h})$. Depending on the efficiency of this
operation, we can start eliminating elements from $A(S, N_k)$ already
at some earlier stage of the computation of~$A(S, N)$, thus reducing the
effort needed for the subsequent stages of the procedure.

If we want to reach a {\em very} large height bound, then we should at
some point switch over to the variant of the sieving procedure described
in Section~\ref{SubS:IntPt} below.

\medskip

Of course, there is a simpler alternative, which is to enumerate all
lattice points in $(J(\Q)/J(\Q)_{\tors}, \hat{h})$ of norm~$\le H'$
and then checking all corresponding points in~$J(\Q)$ whether they are
in the image of~$\iota$. (For this test, one conveniently uses reduction
mod~$p$ again, for a suitable set of primes~$p$.) Which of the two methods
will be more efficient will depend on the curve in question and on the
height bound~$H$. If the curve is fixed, then we expect our Mordell-Weil
sieve method to be more efficient than the short vectors enumeration
when $H$ gets large. The reason for this is that once $S$ and~$N$ are
sufficiently large, the set $A(S, N)$ is expected to be uniformly small
(most of its elements should come from rational points on~$C$), and so the
computation of~$A(S, N)$ for large~$N$ will not take much additional time.
On the other hand, the number of vectors of norm~$\le H'$ will grow
like a power of~$H'$, and the enumeration will eventually become infeasible.


\subsection{Integral Points on Hyperelliptic Curves} \label{SubS:IntPt}

What the preceding application really gives us is a lower bound~$H$
for the logarithmic height of any rational point that we do not know
(and therefore believe does not exist). If we can produce such a bound
in the order of $H = 10^k$ with $k$ in the range of several hundred,
then we can combine
this information with upper bounds for integral points that can be
deduced using linear forms in logarithms and thus determine the set
of integral points on a hyperelliptic curve: if $C : y^2 = f(x)$ is a 
hyperelliptic curve over~$\Q$, then it is possible to compute an upper
bound $\log |x| \le H$ that holds for integral points $(x,y) \in C$,
where $H$ is usually of a size like that mentioned above. See 
Sections 3--9 in~\cite{IntegralPoints}.

With the procedure we have described here, it is feasible to reach
values of~$N$ in the range of~$10^{100}$, corresponding to $H \approx 10^{200}$.
However, this is usually not enough --- the upper bounds provided by the 
methods described in~\cite{IntegralPoints} are more like~$10^{600}$.
The part of the computation that dominates the running time is the
computation of the image of~$C(\F_p)$ in the abstract finite abelian
group representing~$J(\F_p)$. To close the gap, we therefore
switch to a different sieving strategy that avoids having to compute
all these roughly~$p$ discrete logarithms in~$J(\F_p)$. We assume that
we know a subgroup $L \subset J(\Q)$ (initially this is $N J(\Q)$) such
that the image of $C(\Q)$ in~$J(\Q)/L$ is given by rational points we
already know on~$C$. We then try to find a smaller subgroup~$L'$ with
the same property. Let $q$ be a prime of good reduction, and recall
the notation $\phi_q : J(\Q) \to J(\F_q)$ for the reduction homomorphism.
Let $W \subset J(\Q)$ be the image of the known rational points on~$C$,
let $L' = L \cap \ker \phi_q$, and take $R \subset J(\Q)$ to be a complete
set of representatives of the {\em nontrivial} cosets of~$L'$ in~$L$.
We can now check for each $w \in W$ and $r \in R$
whether $\phi_q(w+r) \notin \iota\bigl(C(\F_q)\bigr)$. If this is the
case, then $W$ will also represent the image of~$C(\Q)$ in~$J(\Q)/L'$. 
Note that this test does not require
the computation of a discrete logarithm. We still need to find the
discrete logarithms of the images under~$\phi_q$ of our generators
of the Mordell-Weil group in order to find the kernel of~$\phi_q$,
but this is a small fixed number of discrete log computations for
each~$q$.

The Weil conjectures tell us that
$\#C(\F_q)/\#J(\F_q) \approx 1/q$ when $C$ has genus~$2$, so
the chance that we are successful in replacing $L$ with~$L'$ is in the range of
$(1 - \frac{1}{q})^{((L : L') - 1) \cdot \#W}$. This will be very small when
$(L : L') \cdot \#W$ is much larger than~$q$. Therefore we try to pick~$q$
such that $L/(L \cap \ker \phi_q)$ is nontrivial, but comparable with~$q$
in size. A necessary condition for this is that the part of the group
order of the image of~$\phi_q$ that is coprime to the index of~$L$
in~$J(\Q)$ is $\ll q$. Since it is much faster to compute $\#J(\F_q)$
than it is to compute~$\phi_q$ and its image and kernel, we simply
check~$\#J(\F_q)$ instead. When $q$ passes this test, we do the more
involved computation of the group structure of~$J(\F_q)$ and the
images of the generators of~$J(\Q)$ in the corresponding abstract group,
so that we can find the kernel of~$\phi_q$ and check the condition
on~$(L : L')$. If $q$ passes also this test, we check if we can replace
$L$ by~$L'$. Of course, we can abort this computation (and declare
failure) as soon as we find some $w+r$ as above such that $w+r$ maps
into $\iota\bigl(C(\F_q)\bigr)$.
See Section~11 of~\cite{IntegralPoints}. The idea for this second
sieving stage is due to Samir Siksek.

If $N$ is sufficiently large, then we will have a good chance of finding
enough primes~$q$ that allow us to go to a subgroup of larger index.
Also, once we have been successful with a number of primes, more primes
might become available for future steps, since the index of
$L \cap \ker \phi_q$ in~$L$ may have become smaller.

In the two examples treated in~\cite{IntegralPoints}, this second stage
of the sieving procedure was successful in reaching a subgroup of sufficiently
large index (up to~$10^{1800}$) to be able to conclude that any putative
unknown integral point must be so large as to violate the upper bounds
obtained earlier.


\subsection{Combination with Chabauty's method} \label{SubS:Chab}

Chabauty originally came up with his method in~\cite{Chabauty} in order
to prove a special case of Mordell's Conjecture. More recently, it has
been developed into a powerful tool that allows us in many cases to
determine the set of rational points on a given curve, see for example
\cite{Coleman,FlynnCh,StollTw,McCallumPoonen}.
We can combine it with the Mordell-Weil
sieve idea to obtain a very efficient procedure to determine~$C(\Q)$.
Examples of Chabauty computations supported by sieving can be found
in~\cite{BruinElkies,BruinCompMath,PSS}. In these examples it is the
Chabauty part that is the focus of the computation, and sieving has a
helping role. This is in contrast to what we
describe here, where sieving is at the core of the computation, and the Chabauty
approach is just used to supply us with a `separating' number~$N$
such that $C(\Q)$ injects into $J(\Q)/NJ(\Q)$.

Chabauty's method is applicable when the rank of~$J(\Q)$ is less than
the genus $g$ of~$C$. In this case, for every prime~$p$, there is a regular 
nonzero differential $\omega_p \in \Omega(C_{\Q_p})$ that annihilates the
Mordell-Weil group under the natural pairing 
$J(\Q_p) \times \Omega(C_{\Q_p}) \to \Q_p$. 
If $p$ is a prime of good reduction for~$C$, then a suitable multiple
of~$\omega_p$ reduces mod~$p$ to a nonzero regular differential
$\bar{\omega}_p \in \Omega(C_{\F_p})$.
If $P \in C(\F_p)$ is a point such that $\bar{\omega}_p$ does not
vanish at~$P$ (and $p \ge 3$), then there is at most
one rational point on~$C$ that reduces mod~$p$ to~$P$. 
See for example~\cite[\S\,6]{StollTw}.

On the other hand, if $N$ is divisible by the exponent of $J(\F_p)$,
then the rational points on~$C$ mapping via~$\iota$ into a given coset
of $NJ(\Q)$ in~$J(\Q)$ will all reduce mod~$p$ to the same point
in~$C(\F_p)$. So if $\bar{\omega}_p$ does not vanish at any point 
in~$C(\F_p)$, then we know that each coset of~$NJ(\Q)$ can contain the
image under~$\iota$ of at most one point in~$C(\Q)$. If there is no such
point and we assume the
Main Conjecture of~\cite{StollCov}, then we will be able to show using
the Mordell-Weil sieve that no point of~$C(\Q)$ maps to this coset.
If there is a point, we will eventually find it.

This leads to the following outline of the procedure.
\begin{steplist}{9.}
  \item[1.]
    Find a prime $p \ge 3$ of good reduction for~$C$ such that there
    is $\omega_p \in \Omega(C_{\Q_p})$ annihilating $J(\Q)$ and such that
    $\bar{\omega}_p$ does not vanish on~$C(\F_p)$.
  \item[2.]
    Find a suitable set $S$ of primes and a number~$N$ as described
    in Sections \ref{SubS:S} and~\ref{SubS:N} above, with the additional
    condition that the exponent of~$J(\F_p)$ divides~$N$.
  \item[3.]
    Compute $A(S, N)$ as described in Section~\ref{SubS:A} above.
  \item[4.]
    For each element $a \in A(S, N)$, verify that it comes from a rational
    point on~$C$. To do this, we take the point of smallest canonical
    height in the coset of~$NJ(\Q)$ given by~$a$ and check if it comes
    from a rational point on~$C$. If it does, we record the point.
  \item[5.]
    If the previous step is unsuccessful, we enlarge~$S$ and/or increase~$N$
    and compute a new $A(S, N)$ based on the unresolved members of the
    old $A(S, N)$. We then continue with Step~4.
\end{steplist}

We have implemented this procedure in~{\sf MAGMA} and used it on a large number
of genus~2 curves with Jacobian of Mordell-Weil rank~1. It proved to be
quite efficient: the computation usually takes less than two seconds and
almost always less than five seconds. For this
implementation, we assume that one rational point is already known and use
it as a base-point for the embedding~$\iota$. In practice, this is no
essential restriction, as there seems to be a strong tendency for small
points (which can be found easily) to exist on~$C$ if there are rational 
points at all. Of course, we also need to know a generator of the free part
of~$J(\Q)$, or at least a point of infinite order in~$J(\Q)$. If we only
have a point~$P$ of infinite order, we also have to check that the index of
$\Z \cdot P + J(\Q)_{\tors}$ in~$J(\Q)$ is prime to~$N$. If $P$ is not
a generator, then in Step~4, we could have the problem that the point we
are looking for is not in the subgroup generated by~$P$ (mod torsion).
In this case, the smallest representative of~$a$ is likely to look large,
and we should first try to see if some multiple of~$a$ is small, so that
it can be recognized. A version of this procedure is used by the
{\tt Chabauty} function provided by recent releases of~{\sf MAGMA}.

As mentioned in the discussion above, Steps 4 and~5 will eventually be
successful if the Main Conjecture of~\cite{StollCov} holds for~$C$. There is,
however, an additional assumption we have to make, and that is that Step~1
will always be successful. We state this as a conjecture.

\begin{Conjecture} \label{ConjDiff}
  Let $C/\Q$ be a curve of genus~$g \ge 2$ such that its Jacobian is
  simple over~$\Q$ and such that the Mordell-Weil rank $r$ is less than~$g$. 
  Then there are infinitely many primes~$p$
  such that there exists a regular differential 
  $\omega_p \in \Omega(C_{\Q_p})$ annihilating~$J(\Q)$ such that the
  reduction mod~$p$ of (a suitable multiple of)~$\omega_p$ does not
  vanish on~$C(\F_p)$.
\end{Conjecture}

Of course, this can easily be generalized to number fields in place of~$\Q$.

We need to assume that the Jacobian is simple, since otherwise there can
be a differential killing the Mordell-Weil group that comes from one
of the simple factors. Such a differential can possibly
vanish at a rational point on the curve, and then its reductions mod~$p$ will vanish 
at an $\F_p$-point for all~$p$. For example, when $C$ is a curve of
genus~2 that covers two elliptic curves, one of rank zero and one of
rank~1, then the (essentially unique) differential killing the Mordell-Weil
group will be the pull-back of the regular differential on one of the
elliptic curves, hence will be a global object. Of course, in such a
case, we can instead work with one of the simple factors that still
satisfies the `Chabauty condition' that its Mordell-Weil rank is less
than its dimension.

We give a heuristic argument that indicates that Conjecture~\ref{ConjDiff}
is plausible. We first prove a lemma.

\begin{Lemma} \label{LemmaDiff}
  Let $C$ be a smooth projective curve of genus~$g \ge 2$ over~$\F_p$.
  The probability
  that a random nonzero regular differential $\bar\omega$ on~$C$
  does not vanish on~$C(\F_p)$ is at least $\frac{1}{3} + O(g p^{-1/2})$.
\end{Lemma}

\begin{proof}
  First assume that $C$ is not hyperelliptic.
  Then we can consider the canonical embedding $C \to \BP^{g-1}$.
  We have to estimate the number~$n$ of hyperplane sections that do not 
  meet the image of~$C(\F_p)$. If $g = 3$, then $C \subset \BP^2$ is a smooth plane
  quartic curve, and the nonzero regular differentials correspond to
  $\F_p$-defined lines in~$\BP^2$ (up to scaling). Let $\ell_k$ ($k = 0, 1, 2, 4$)
  be the number of such lines that contain exactly $k$ points of~$C(\F_p)$
  (with multiplicity). We want to estimate~$\ell_0$. In the following, we disregard
  lines that are tangent to~$C$ in an $\F_p$-rational point; their number
  is~$O(p)$ and so the result is unaffected by them. 
  
  Fix a point $P \in C(\F_p)$ and consider the $(p+1)$ lines through~$P$.
  Projection away from~$P$ gives a covering $C \to \BP^1$ of degree~3,
  which can be Galois only for at most four choices of~$P$ (since a
  necessary condition is that five tangents at inflection points of~$C$
  meet at~$P$, and there are at most~24 such tangents). These potential
  exceptions do not affect our estimate. For the other points, the
  covering has Galois group~$S_3$, and by results in~\cite{MurtyScherk},
  we have, denoting by $\ell_{k,P}$ the number of lines through~$P$ meeting
  $C(\F_p)$ in exactly~$k$ points:
  \[ \ell_{1,P} = \frac{p}{3} + O(\sqrt{p}), \quad
     \ell_{2,P} = \frac{p}{2} + O(\sqrt{p}) \quad\text{and}\quad
     \ell_{4,P} = \frac{p}{6} + O(\sqrt{p}).
  \]
  We obtain
  \begin{align*}
    \ell_1 &= \hphantom{\frac{1}{2}}
              \sum_P \ell_{1,P} + O(p) = \frac{p^2}{3} + O(p^{3/2}) \\
    \ell_2 &= \frac{1}{2} \sum_P \ell_{2,P} + O(p) = \frac{p^2}{4} + O(p^{3/2}) \\
    \ell_4 &= \frac{1}{4} \sum_P \ell_{4,P} + O(p) = \frac{p^2}{24} + O(p^{3/2}) \\
    \ell_0 &= p^2+p+1 - (\ell_1+\ell_2+\ell_4) = \frac{3}{8} p^2 + O(p^{3/2})\,,
  \end{align*}
  which shows that the probability here is $\frac{3}{8} + O(p^{-1/2})$.
  
  Now let $g \ge 4$ (still assuming that $C$ is not hyperelliptic).
  Let $t_3$ denote the number of triples of distinct points in~$C(\F_p)$
  that are collinear in the canonical embedding.
  By the inclusion-exclusion principle, we have for the number~$n$
  of hyperplane sections missing~$C(\F_p)$
  \begin{align*}
    n &\ge \#\BP^{g-1}(\F_p) - \#C(\F_p) \#\BP^{g-2}(\F_p)
            + \binom{\#C(\F_p)}{2}\,\#\BP^{g-3}(\F_p) \\
      &\quad{} 
        - \Bigl(\binom{\#C(\F_p)}{3} - t_3\Bigr)\,\#\BP^{g-4}(\F_p)
        - t_3 \#\BP^{g-3}(\F_p) \,.
  \end{align*}
  A collinear triple is part of a one-dimensional linear system of
  degree~3 on~$C$. It is known that there are at most two such linear systems
  when $g = 4$ (see, e.g., \hbox{\cite[Example~IV.5.5.2]{Hartshorne}}) and at most
  one when $g \ge 5$ (see, e.g.,~\cite[Example~I.3.4.3]{EMS}). This implies
  that $t_3 \le 2(p+1)$, and therefore that $t_3$ has no effect on the
  estimate below. Since $\#C(\F_p) = p + O(g p^{1/2})$, we find that
  \[ \frac{n}{\#\BP^{g-1}(\F_p)}
      \ge 1 - 1 + \frac{1}{2} - \frac{1}{6} + O(g p^{-1/2}) 
      = \frac{1}{3} + O(g p^{-1/2}) \,.
  \]
  If $C$ is hyperelliptic, the problem is equivalent to
  the question, how likely is it for a random homogeneous polynomial
  of degree~$g-1$ in two variables not to vanish on the image $X$ of~$C(\F_p)$
  in~$\BP^1(\F_p)$ under the hyperelliptic quotient map $C \to \BP^1$?
  The number~$n$ in this case can be estimated by
  \[ n \ge \#\BP^{g-1}(\F_p) - \#X\#\BP^{g-2}(\F_p) \]
  Since the size of~$X$ is $p/2 + O(g p^{1/2})$, we obtain here even
  \[ \frac{n}{\#\BP^{g-1}(\F_p)}
      \ge 1 - \frac{1}{2} + O(g p^{-1/2}) 
      = \frac{1}{2} + O(g p^{-1/2}) \,.
  \]
\end{proof}

We expect that arguments similar to that used in the non-hyperelliptic
genus~3 case can show that the probability in question is
\[ \alpha_g + O_g(p^{-1/2}) \quad\text{with}\quad
   \alpha_g = \sum_{k=0}^{2g-2} \frac{(-1)^k}{k!} \approx e^{-1}
\]
in the non-hyperelliptic case. In the hyperelliptic case, the corresponding
probability
\[ \beta_g + O_g(p^{-1/2}) \quad\text{with}\quad
   \beta_g = \sum_{k=0}^{g-1} \frac{(-1)^k}{2^k k!} \approx e^{-1/2}
\]
is obtained by an obvious extension of the argument used in the proof above.

We now consider a curve $C/\Q$ as in Conjecture~\ref{ConjDiff}, with
$r = g-1$. It seems reasonable to assume that the reduction $\bar\omega_p$
of the unique (up to scaling) differential $\omega_p$ annihilating~$J(\Q)$
behaves like a random
element of $\Omega^1(C/\F_p)$ as $p$ varies. By Lemma~\ref{LemmaDiff},
we would then expect even a set of primes~$p$ of positive density $\ge 1/3$
such that $\bar\omega_p$ does not vanish on~$C(\F_p)$.

When $r \le g-2$, the situation should be much better. We have at least
a pencil of differentials, giving rise to a linear system of degree~$2g-2$
and positive dimension on the curve over~$\F_p$. Unless this linear
system has a base-point in~$C(\F_p)$, effective versions of the Chebotarev density
theorem as in~\cite{MurtyScherk} show that there is a divisor
in the system whose support does not contain rational points, at least
when $p$ is sufficiently large. However, we still have to exclude the
possibility that the relevant linear system has a base-point in~$C(\F_p)$
for (almost) every~$p$.

If we mimick the set-up of Lemma~\ref{LemmaDiff} in the situation when
$g - r = d \ge 2$, then we have to look at the Grassmannian of
$(r-1)$-dimensional linear subspaces in~$\BP^{g-1}$: there is a
$d$-dimensional linear space of differentials killing~$J(\Q)$, and
the intersection of the corresponding hyperplanes in~$\BP^{g-1}$
is an $(r-1)$-dimensional (projective) linear subspace.
The set of such subspaces through a given point corresponds via
projection away from this point to~$\Grass(\BP^{r-2} \subset \BP^{g-2})$,
so by the simplest case of the inclusion-exclusion
inequality, we have for the number~$n$ of base-point free subspaces:
\[ n \ge \#\Grass(\BP^{r-1} \subset \BP^{g-1})
            - \#C(\F_p)\,\#\Grass(\BP^{r-2} \subset \BP^{g-2}) \,,
\]
and therefore a `density' of
\[ \frac{n}{\#\Grass(\BP^{r-1} \subset \BP^{g-1})}
     \ge 1 - \#C(\F_p)\, \frac{\#\Grass(\BP^{r-2} \subset \BP^{g-2})}%
                              {\#\Grass(\BP^{r-1} \subset \BP^{g-1})}
     = 1 - O(p^{-(d-1)}) \,.
\]
When $d = 2$, one is thus led to expect an infinite but very sparse set of
primes such that there is a base-point (since $\sum p^{-1}$ diverges),
whereas for $d > 2$, one would expect only finitely many such primes.

If we modify the algorithm in such a way that it considers (arbitrarily) `deep'
information at~$p$, then the requirement can be weakened to the following.

\begin{Conjecture} \label{Conjp}
  Let $C/\Q$ be a curve of genus~$g \ge 2$ such that its Jacobian is 
  simple and has Mordell-Weil rank $r < g$. Then there is a prime $p \ge 3$
  such that there exists a regular nonzero differential 
  $\omega_p \in \Omega(C_{\Q_p})$ annihilating~$J(\Q)$ such that 
  $\omega_p$ does not vanish on~$C(\Q)$.
\end{Conjecture}

Heuristically, the probability that $\omega_p$ does vanish at a rational
point should be zero (except when there is a good reason for it, see above),
which lets us hope that the weaker conjecture may be amenable to proof.
In fact, Tzanko Matev (a PhD student of Michael Stoll) has
recently established a $p$-adic version of the
`analytic subgroup theorem' for abelian varieties (see~\cite{BW} for the background).
It states that when $J$ is absolutely simple, then
the $p$-adic logarithm of an algebraic point on~$J$
cannot be contained in a proper subspace of the tangent space $T_0 J(\Q_p)$
that is generated by algebraic vectors.
This implies that the statement of Conjecture~\ref{Conjp} is true
for every~$p$ when the Mordell-Weil rank is~$1$.


\section{Information at bad primes} \label{S:bad}

This and the following section discuss how to
extract the information that the Mordell-Weil sieve needs as input
in the specific case that $C$ is a curve of genus~2 over~$\Q$ 
(or a more general number field) and
we are not just interested in $C(\F_p$) and $J(\F_p)$ for a prime~$p$
of good reduction.

\medskip

In particular
when the rank is large, which in practice means $r \ge 3$, it becomes
important to use sufficient `local' information to keep the sizes of
the sets~$A(S, N_j)$ reasonably small. A valuable source of such information
is given by primes of bad reduction, as the group orders of suitable
quotients of $J(\Q_p)$ tend to be rather smooth. More precisely,
we would like to make use of the top layers of the filtration given
by the well-known exact sequences
\[ 0 \To J^0(\Q_p) \To J(\Q_p) \To \Phi_p(\F_p) \To 0 \]
and
\[ 0 \To J^1(\Q_p) \To J^0(\Q_p) \To \tilde{J}(\F_p) \To 0 \,. \]
Here $\Phi_p$ is the component group of the special fiber of the
N\'eron model of~$J$ over~$\Z_p$ and $\tilde{J}$ is the connected
component of the special fiber (and $J^1(\Q_p)$ is the kernel of reduction).

In this section, we describe how this information can be obtained when
$C$ is a genus~$2$ curve, $p$ is odd, and the given model of~$C$
is regular at~$p$. Here and in the following, we will use $J^1(\Q_p)$,
and later $J^n(\Q_p)$, to denote the kernel of reduction and the
`higher' kernels of reduction {\em with respect to the given model
of the curve}. If the model is not minimal in a suitable sense, then
our kernel of reduction will be strictly contained in the kernel of
reduction with respect to a N\'eron model. To be precise, for us,
$J^1(\Q_p)$ denotes the subgroup of points in~$J(\Q_p)$ whose reduction
mod~$p$ on the projective model in~$\BP^{15}$ given as described
in~\cite[Ch.~2]{CasselsFlynn} is the origin; see below. Of course,
this then changes the meaning of the quotients in the sequences above.

\medskip

But first, we will establish some general facts.
Let $k$ be a field with $\Char(k) \neq 2$, and let
\[ F(X,Z) = f_6 X^6 + f_5 X^5 Z + f_4 X^4 Z^2 + f_3 X^3 Z^3 + f_2 X^2 Z^4
             + f_1 X Z^5 + f_0 Z^6
\]
be a homogeneous polynomial of degree~6 with coefficients in~$k$. We do
not assume that $F$ is squarefree or even that $F \neq 0$. 
\begin{Definition} \strut
  \begin{enumerate}\addtolength{\itemsep}{1mm}
    \item Let $C_F$ be the curve given by the equation
          \[ Y^2 = F(X,Z) \]
          in the weighted projective plane with weights $(1, 3, 1)$ for the
          coordinates $X, Y, Z$, respectively.
    \item Denote by $J_F$ the scheme in~$\BP^{15}_k$ that is defined by the
          72 quadrics described in~\cite[Ch.~2]{CasselsFlynn}
          (see~\cite[jacobian.variety/defining.equations]{FlynnFTP}
          for explicit equations).
    \item Let $K_F$ be the surface in~$\BP^3_k$ that is defined by the
          Kummer surface equation as given in~\cite[Ch.~3]{CasselsFlynn},
          and denote by $\delta_F = \delta = (\delta_{1}, \dots, \delta_{4})$
          the polynomials giving the duplication map on the Kummer surface,
          see~\cite[kummer/duplication]{FlynnFTP}.
          
          \smallskip
          \noindent
          If $\delta(P) \neq 0$, then we write $\BP \delta(P) \in \BP^3$ for
          the point with projective coordinates $(\delta_1(P) : \ldots : \delta_4(P))$.
    \item Let $\tilde{D}_F \subset \BA^3_k \times \BA^4_k \times \BA^5_k$
          be the scheme of triples $(A, B, C)$ such that $A \neq 0$ and
          \[ F(X, Z) = A(X, Z) C(X, Z) + B(X, Z)^2 \,, \]
          where we set for $A = (a_0, a_1, a_2)$, $B = (b_0, \ldots, b_3)$
          and $C = (c_0, \dots, c_4)$
          \begin{align*}
            A(X, Z) &= a_2 X^2 + a_1 X Z + a_0 Z^2 \,, \\
            B(X, Z) &= b_3 X^3 + b_2 X^2 Z + b_1 X Z^2 + b_0 Z^3 \,, \\
            C(X, Z) &= c_4 X^4 + c_3 X^3 Z + c_2 X^2 Z^2 + c_1 X Z^3 + c_0 Z^4 \,.
          \end{align*}
          Let $D_F \subset \BP^2_k \times \BA^4_k$ be the image of~$\tilde{D}_F$
          under the projection to the first two factors, followed by the
          canonical map $\BA^3 \setminus \{0\} \to \BP^2$ on the first factor.
  \end{enumerate}
\end{Definition}

When $F$ is squarefree, then $C_F$ is a smooth curve of genus~2, $J_F$ is
its Jacobian, and $K_F$ is the associated Kummer surface. The scheme~$D_F$
then gives the possible Mumford representations of effective divisors of
degree~2 on~$C_F$; it therefore maps onto $J_F \setminus \{O\}$.
We will extend these relations to our more general setting.

The `origin' $O = (1:0:\ldots:0)$ is always a (smooth) point on~$J_F$.
The 16 coordinates on~$J_F$ split into 10 `even' and 6 `odd' ones;
the even coordinates are given (up to a simple invertible linear
transformation) by the monomials of degree~2 in the coordinates on~$K_F$.


Let us first look at the relation between $J_F$ and~$K_F$.

\begin{Lemma} \label{Lemma:genJK} \strut
  Projection to the ten even coordinates gives rise to a morphism
  $\kappa : J_F \to K_F$, which is a double cover.
\end{Lemma}

\begin{proof}
  The monomials of degree~2 in the odd coordinates can be expressed
  as quadratic forms in the even coordinates. So if all the even
  coordinates vanish, the odd coordinates have to vanish, too.
  Therefore projection to the~$\BP^9$ spanned by the even coordinates
  is a morphism. The relations between the even coordinates are
  exactly those coming from the fact that the even coordinates
  come from the monomials of degree~2 in the coordinates of the~$\BP^3$
  containing~$K_F$, together with the quadratic relation coming
  from the quartic equation defining~$K_F$. Therefore the image
  of~$J_F$ in~$\BP^9$ is the image of~$K_F$ under the 2-uple embedding
  of~$\BP^3$ into~$\BP^9$ and therefore isomorphic to~$K_F$. This gives the
  morphism~$\kappa$. The fact stated in the first sentence of this proof
  then implies that $\kappa$ is a (ramified) double cover.
\end{proof}

Now let us consider the relation between $\tilde{D}_F$, $D_F$ and~$J_F$.

\begin{Lemma} \label{Lemma:DtoJ}
  There is a morphism
  \[ \phi : D_F \to J_F \setminus \{O\} \]
  that specializes to the representation of points on~$J_F$ mentioned
  above when $F$ is squarefree. The morphism~$\phi$ is surjective
  on $k$-points and makes the following diagram commute:
  \[ \xymatrix{ D_F \ar[r]^(0.4){\phi} \ar@/_/[rdd]_{\mathrm{pr}_1}
                 & J_F \setminus \{O\} \ar[d]^{\kappa} \\
                 & K_F \setminus \{\kappa(O)\} \ar[d]^{\mathrm{pr}'} \\
                 & \BP^2
              }
     \begin{array}[t]{cccc}
       \mathrm{pr}' : & K_F & \To & \BP^2 \\
                  & (x_1:x_2:x_3:x_4) & \longmapsto & {(x_3:-x_2:x_1)} \\[1em]
       \mathrm{pr}_1 : & \BP^2 \times \BA^3 & \To & \BP^2 \\
                  & ((a_0:a_1:a_2), B) & \longmapsto & (a_0:a_1:a_2)
     \end{array}
  \]
  Furthermore, $\phi(A, B) = \phi(A', B')$ if and only if $A = A'$ and
  $B(X,Z) \equiv B'(X,Z) \bmod A(X,Z)$.
\end{Lemma}

\begin{proof}
  Let $(A, B) \in D_F$. Then $\phi$ can be given as
  \begin{align*}
    \phi(A, B)
      &= \bigl(* : * : * : * : * : * \\
      & \qquad {} : -b_2 a_0^2 + b_1 a_0 a_1 - b_0 (a_1^2 - a_0 a_2)
                 :  b_3 a_0^2 - b_1 a_0 a_2 + b_0 a_1 a_2 \\
      & \qquad {} : -b_3 a_0 a_1 + b_2 a_0 a_2 - b_0 a_2^2
                 :  b_3 (a_1^2 - a_0 a_2) - b_2 a_1 a_2 + b_1 a_2^2 \\
      & \qquad {} : a_0^2 : -a_0 a_1 : a_0 a_2 : -a_1 a_2 : a_2^2 : a_1^2 - 4 a_0 a_2
              \bigr) \,.
  \end{align*}
  One can check using the defining equations of~$J_F$ given at~\cite{FlynnFTP}
  that the first six coordinates are uniquely determined by
  the last ten when the last six are not all zero. It is also possible to
  write down expressions for the first six coordinates in terms of $A$, $B$
  and $C$, where $(A, B, C)$ is the point on~$\tilde{D}_F$ mapping to~$(A, B)$.
  The image of the point above under~$\kappa$ has the form $(a_2 : -a_1 : a_0 : *)$, 
  which shows that
  $\pr' \circ \kappa \circ \phi = \pr_1$. It remains to show that $\phi$ is
  surjective on $k$-points. Let $P \in J_F(k) \setminus \{O\}$, then
  $A = \pr'(\kappa(P)) \in \BP^2(k)$ is defined. Consider the middle four
  coordinates on~$J$ (n\textsuperscript{\,\underline{os}}~7 through~10).
  The expression for~$\phi(A,B)$
  given above gives rise to a system of linear equations for~$B$. The last
  six of the equations defining~$J_F$ ensure that the system has
  a solution~$B \in \BA^4(k)$. Then $\phi(A,B)$ agrees with~$P$ in the
  last ten coordinates; therefore we must have $\phi(A,B) = P$.
  
  To show the last statement, note first that $\phi(A, B) = \phi(A', B')$
  implies $A = A'$ (apply $\pr' \circ \kappa$). The kernel
  of the matrix giving the linear equations determining~$B$ is spanned
  by the coefficient tuples of $Z A(X,Z)$ and $X A(X,Z)$. This shows
  that $\phi(A, B) = \phi(A, B') \iff A(X,Z) \mid B(X,Z) - B'(X,Z)$.
\end{proof}

By the above, the fibers of the map $\phi : D_F \to J_F \setminus \{O\}$
are isomorphic to~$\BA^2$. We can remove this ambiguity at the cost of
restricting to a subscheme.

\begin{Lemma} \label{Lemma:Uj}
  Let
  \begin{align*}
    U_0 &= \{(A,B) \in D_F : a_0 = 1, b_0 = b_1 = 0\}\,, \\
    U_1 &= \{(A,B) \in D_F : a_1 = 1, a_0 a_2 \neq 1, b_1 = b_2 = 0\}\,, \\
    U_2 &= \{(A,B) \in D_F : a_2 = 1, b_2 = b_3 = 0\} \,.
  \end{align*}
  Then $\phi|_{U_j}$ is an isomorphism onto its image for each $j \in \{0,1,2\}$,
  and 
  \[ \phi(U_0) \cup \phi(U_1) \cup \phi(U_2) = J_F \setminus \{O\} \,. \]
\end{Lemma}

\begin{proof}
  In each case, the linear system giving $b_0, \dots, b_3$
  in terms of the middle four coordinates on~$J_F$, together with the
  conditions $b_j = b_{j+1} = 0$ has a unique solution, giving the inverse
  morphism $\phi(U_j) \to U_j$. The last statement then follows, since
  the images of the~$U_j$ in~$\BP^2$ cover~$\BP^2$.
\end{proof}

Now we can describe the smooth locus of~$J_F$.

\begin{Proposition} \label{Prop:Jsmooth}
  The origin $O$ is always a smooth point on~$J_F$. If $P \in J_F \setminus \{O\}$,
  write $P = \phi(A,B)$ with $(A,B) \in D_F$. Then $P$ is a singular point
  on~$J_F$ if and only if
  \begin{enumerate}\addtolength{\itemsep}{1mm}
    \item $A(X,Z)$ has a simple root (in~$\BP^1$) at a multiple root of~$F$, or
    \item $A(X,Z) = cL(X,Z)^2$ has a double root at a multiple root of~$F$ and
          $L(X,Z)^3$ divides $F(X,Z) - B(X,Z)^2$.
  \end{enumerate}
\end{Proposition}

Note that the last condition means that the curve $Y = B(X,Z)$ is tangential
to a branch of~$C_F$ at the singular point $L(X,Z) = Y = 0$.

\begin{proof}
  The statement that $O \in J_F$ is smooth is easily checked using the
  explicit equations.
  The general statement is geometric, so we can assume $k$ to be algebraically closed.
  Then there is a transformation $\sigma \in \GL_2(k)$ such that
  $A^\sigma(X,Z) = XZ$ or~$X^2$. In the first case, we can take $Q \in U_1$,
  and we easily check that $Q$ is singular on~$U_1$ if and only if $f_0 = f_1 = 0$
  or $f_6 = f_5 = 0$, which means that $F$ has a multiple root at one of the
  two simple roots of~$A(X,Z)$, namely $0$ or~$\infty$. In the second case,
  we can take $Q \in U_2$, and we find that $Q$ is singular on~$U_2$ if and
  only if $f_0 = f_1 = f_2 - b_1^2 = 0$, which means that $F$ has a multiple
  root at the double root~$0$ of~$A(X,Z)$ and that $X^3 = L(X,Z)^3$ divides
  $F(X,Z) - B(X,Z)^2$. Since $\phi$ is an isomorphism on~$U_j$, $P$ is singular
  on~$J_F$ if and only if $Q$ is singular on~$U_j$.
\end{proof}

\begin{Definition}
  We denote by $D'_F$ the locus of points $Q \in D_F$ such that $\phi(Q)$
  is a smooth point on~$J_F$, and we write $J'_F$ for the subscheme of
  smooth points on~$J_F$.
\end{Definition}

According to Prop.~\ref{Prop:Jsmooth} above, the complement of $D'_F$ in~$D_F$
consists of the points $(A, B)$ satisfying one of the conditions in the proposition.

\begin{Lemma} \label{Lemma:Jirr}
  Assume that $k$ is algebraically closed. Then
  $J_F$ is reduced and irreducible except in the following two cases.
  \begin{enumerate}\addtolength{\itemsep}{1mm}
    \item $F = 0$. Then $J_F$ has two irreducible components. One is 
          $\phi(\BP^2 \times \{0\})$ and is not reduced,
          the other contains~$O$, and its
          remaining points are of the form $\phi(A, B)$ such that there
          is a linear form~$L$ with $A(X,Z) = cL(X,Z)^2$ and $L(X,Z) \mid B(X,Z)$.
    \item $F = H(X,Z)^2$ is a nonzero square. Then $J_F$ has three irreducible
          components, all of which are reduced.
          Two of them are given by $\phi(\BP^2 \times \{\pm H\})$,
          the third contains the origin~$O$.
  \end{enumerate}
\end{Lemma}

\begin{proof}
  It is easy to check the claim in the two special cases. In all other cases,
  $C_F$ is reduced and irreducible. Consider the symmetric square $C_F^{(2)}$.
  Let $S \subset C_F$ be the (finite) set of singular points (given by the
  mutiple roots of~$F$) . Identify $S$ with its image in~$C_F^{(2)}$ under
  the diagonal map. There is a morphism
  \[ \psi : C_F^{(2)} \setminus S \to J_F \]
  that can be defined using the expressions for the coordinates on the Jacobian
  given in~\cite[Ch.~2]{CasselsFlynn}. Its image is
  \[ J_F \setminus \{\phi(A,B) : A(X,Z) = cL(X,Z)^2,
                                 \text{$L(P) = 0$ for some $P \in S$}\} \,,
  \]
  which is dense in~$J_F$. Since $C_F^{(2)}$ is irreducible, this implies
  that $J_F$ is irreducible as well. The component containing the origin
  is always reduced, since the origin is a smooth point.
\end{proof}

\begin{Remark}
  If $k$ is not algebraically closed, then there is the additional case
  $F = c H(X,Z)^2$ with $H \neq 0$ and a non-square $c \in k$.
  According to Lemma~\ref{Lemma:Jirr},
  $J_F$ has three geometric components. One is defined over~$k$ and contains
  the origin, the other two are conjugate over~$k(\sqrt{c})$ and do not
  have any smooth $k$-points.

  If we apply the argument used in the proof above
  in the case $F = H^2 \neq 0$, then $C_F$ has two
  components, therefore $C_F^{(2)}$ has three, and we see again that~$J_F$
  has three (reduced) irreducible components.
  
  From the description given in the proof, we see that $\psi$ extends to
  a morphism
  \[ \tilde{\psi} : \Bl'_S C_F^{(2)} \To J_F \,. \]
  Here $\Bl'_S C_F^{(2)}$ is obtained from~$C_F^{(2)}$ by replacing
  each point in~$S$ by a~$\BP^1$ in such a way that locally near
  a point in~$S$, $\Bl'_S C_F^{(2)}$ is the closure of the graph of the
  rational map giving the slope (in a suitable affine chart)
  of the line connecting the two points
  in the divisor corresponding to a point in~$C_F^{(2)}$.
  Let $\pi : C_F \to \BP^1$ be the canonical map, and denote by~$\pi^*$
  the induced map $\BP^1 \to C_F^{(2)}$. Then $\tilde{\psi}$ is an
  isomorphism away from $\pi^*(\BP^1)$ and contracts $\pi^*(\BP^1)$
  to the origin $O \in J_F$. We therefore have an isomorphism
  \[ \Bl'_S C_F^{(2)} \cong \Bl_O J_F \,. \]
  This generalizes the standard fact that $C_F^{(2)} \cong \Bl_O J_F$
  if $C_F$ is smooth.
\end{Remark}

\begin{Definition}
  We denote by $J^0_F$ the component of the smooth part~$J'_F$ of~$J_F$ that
  contains the origin~$O$. 
  We write $K^0_F$ for the open subscheme of~$K_F$ on which $\delta \neq 0$.
  Let $B_F$ denote the matrix of biquadratic forms
  as defined in~\cite[Ch.~3]{CasselsFlynn};
  see~\cite[kummer/biquadratic.forms]{FlynnFTP} for explicit expressions.
\end{Definition}

\begin{Proposition} \label{Prop:J0K0}
  We have $\kappa(J^0_F) = K^0_F$. Equivalently, a point $P \in J_F$ is smooth
  and on the component of the origin if and only if $\delta(\kappa(P)) \neq 0$.
\end{Proposition}

\begin{proof}
  We can again assume that $k$ is algebraically closed and that
  $\pr'(\kappa(P))$ is one of $(0 : 1 : 0)$ or $(1 : 0 : 0)$. ($O$ is always
  smooth, and $\delta_4(\kappa(O)) \neq 0$.) We represent $P$ as $\phi(Q)$
  with $Q = ((0 : 1 : 0), (b_0, 0, 0, b_3))$ or $Q = ((1 : 0 : 0), (b_0, b_1, 0, 0))$,
  respectively. Then we can use the description of singular points
  given in Prop.~\ref{Prop:Jsmooth} and the description of the components
  of~$J_F$ given in Lemma~\ref{Lemma:Jirr}.
  Writing down the polynomials $\delta_j$ evaluated at
  $\kappa(\phi(Q))$, we conclude after some fairly straightforward
  manipulations that in the first case, $\delta = 0$ if and only if
  $f_0 = f_1 = 0$ or $f_6 = f_5 = 0$, or there are $b_1$, $b_2$ such that
  $F = (b_3 X^3 + b_2 X^2 Z + b_1 X Z^2 + b_0 Z^3)^2$. The first two conditions
  mean as before that there is a singularity at $0$ or~$\infty$, and the
  third says that $P$ is not on the right component. In the second case,
  we find in a similar way that $\delta = 0$ if and only if $X^2 \mid F$
  and $X^3 \mid F - (b_1 X Z^2 + b_0 Z^3)^2$, or $F$ is a square and does
  not vanish at~$0$. The first condition means that $P$ is not smooth, the
  second says again that $P$ is not on the right component.
\end{proof}

This result is due (with a different proof) to Jan Steffen M\"uller,
a PhD student of one of us (Stoll).

Now we can state and prove the main result of this section.

\begin{Theorem} \label{Thm:J0GRoup}
  The scheme $J^0_F$ is a commutative algebraic group in a natural way.
  If we represent its nonzero elements by pairs $(A,B) \in D^0(F)$, then
  composition in the group can be performed by Cantor composition and
  reduction~\cite{Cantor}, except when both polynomials $A(X,Z)$ vanish
  at the same singular point of~$C_F$. Without loss of generality, this
  point is at $X = 0$; then we have
  \[ \phi(X^2, \lambda X Z^2) + \phi(X^2, \mu X Z^2)
        = \phi\Bigl(X^2, \frac{f_2+\lambda\mu}{\lambda+\mu} X Z^2\Bigr)
  \]
  where $F(X,Z) = f_2 X^2 Z^4 + f_3 X^3 Z^3 + \dots + f_6 Z^6$. If
  $\lambda + \mu = 0$, the result is the zero element in~$J^0_F$.
\end{Theorem}

\begin{proof}
  Let $\CO$ be a complete discrete valuation ring with uniformizer~$\pi$, residue
  field~$k$ and field of fractions~$L$. We can then find a homogeneous polynomial
  $\tilde{F} \in \CO[X,Z]$ of degree six that is squarefree and
  whose reduction mod~$\pi$ is~$F$. We denote reduction mod~$\pi$ by a bar.
  Let $G = J_{\tilde F}(L)$, $G^0 = \{P \in G : \bar{P} \in J^0_F(k)\}$,
  and $G^1 = \{P \in G : \bar{P} = O \in J_F(k)\}$. Then for $P \in G^0$
  and $Q \in G^1$, we have $\overline{P+Q} = \bar{P}$. To see this, note
  that the images of $P \pm Q$ under~$\kappa$ are given by $B_{\tilde F}(P,Q)$.
  Since 
  $\overline{B_{\tilde F}(P,Q)} = B_F(\bar{P}, \bar{Q}) \sim \bar{P}^\top \bar{P}$
  (abusing notation by letting $\bar{P}$ denote a vector of projective coordinates
  for~$\bar{P}$), we must have $\kappa(\overline{P \pm Q}) = \kappa(\bar{P})$.
  This implies that $\overline{P + Q} = \bar{P}$ or $\overline{-P}$.
  The function $Q \mapsto \overline{P+Q}$ cannot take exactly two distinct
  values on the residue class of~$O$, so we must have $\overline{P+Q} = \bar{P}$.
  
  This implies that $G^1$ is a subgroup of $G$, that $G^1$ acts on~$G^0$
  and that (at least as sets) $G^0/G^1 \cong J^0_F(k)$.
  By a similar argument, we see that
  $G^0$ is also a subgroup of~$G$ (if $P, Q \in G^0$, then by 
  Prop.~3.1 of~\cite{StollH2}, $B_F(\bar{P},\bar{Q}) \neq 0$, which implies 
  by Lemma~3.2 of~\cite{StollH2} that $\overline{P \pm Q} \in J^0_F(k)$).
  This already shows that $J^0_F(k)$ has a group structure (and the same is
  true for $J^0_F(\ell)$ for every field extension $\ell$ of~$k$).
  
  To see that the group law on $J^0_F$ is given by Cantor's algorithm,
  we can lift two given elements to $G^0$ in such a way that
  we stay in the same case in the algorithm, then apply the algorithm over~$L$
  (in fact, over~$\CO$) and reduce mod~$\pi$. This works unless we are in
  the special case mentioned in the statement of the proposition. The formula
  in this case can be obtained by a suitable limit argument.
  This then also shows that $J^0_F$ is an algebraic group.
\end{proof}

The upshot of this result is that {\em we can do computations} in the
group~$J^0_F(k)$, much in the same way as we compute in the Jacobian
of~$C_F$ when $C_F$ is smooth.

\begin{Remark}
  If $k = \F_q$ and $q$ is odd, one can work out the order of the group 
  $J^0_F(k)$, depending on the factorization of~$F$. This leads to the table in 
  Figure~\ref{table}. The subscripts 
  give the degrees of the factors, which are assumed to be irreducible
  if they occur with multiplicity $> 1$, and to be pairwise coprime.
  $E$ is the genus~$1$ curve $y^2 = h_4(x,1)$ or $y^2 = h_3(x,1)$.
  `sq$(c)$' means that $c$ is a square in~$\F_q^\times$.
  
  If $F = H^2$ is a nonzero square, then by Lemma~\ref{Lemma:Jirr} $J_F$
  splits into three components, the two components not containing~$O$
  being given by $\phi(\BP^2 \times \{\pm H\})$.
  We denote their intersection with~$J'_F$ by $J^\pm_F$.
  In a similar way as above for the group structure of~$J^0_F$,
  we obtain well-defined maps
  \[ J^0_F \times J^+_F \to J^+_F, \quad J^0_F \times J^-_F \to J^-_F,
     \quad\text{and}\quad J^+_F \times J^-_F \to J^0_F
  \]
  that are compatible with the group structure of~$J^0_F$ and show that
  $J^+_F$ and~$J^-_F$ are principal homogeneous spaces under~$J^0_F$.
  Therefore the number of smooth points in~$J_F(k)$ is three times the cardinality
  of~$J^0_F(k)$. On the
  other hand, our addition is not defined on $J^+_F \times J^+_F$ or
  $J^-_F \times J^-_F$. (In this case, the $B$ polynomial one obtains
  in Cantor's algorithm vanishes along one of the components of~$C_F$,
  and we get an undefined~$A$.)
\end{Remark}

\begin{figure}[htb]
\[ \begin{array}{|c|c|c|c|}
     \text{factorization} & c & \text{order if $\operatorname{sq}(c)$} &
                                \text{order otherwise} \\\hline
     0 & - & q^2 & \\
     \ell_1^2 h_4 & \Res(\ell_1, h_4) & (q-1)\,\#E(\F_q)
                                               & (q+1)\,\#E(\F_q) \\
     \ell_1^3 h_3 & - & q\,\#E(\F_q) & \\
     \ell_1^2 m_1^2 h_2
       & \begin{array}{r@{}c@{}l}
           c  &=& \Res(\ell_1, h_2) \\
           c' &=& \Res(m_1, h_2)
         \end{array} 
         & \begin{cases} (q-1)^2 & \text{if $\operatorname{sq}(c')$} \\
                         q^2-1 & \text{else}
           \end{cases}
         & \begin{cases} q^2-1 & \text{if $\operatorname{sq}(c')$} \\
                         (q+1)^2 & \text{else}
           \end{cases} \\
      g_2^2 h_2 & \Res(g_2, h_2)
                & q^2 - 1 & q^2 + 1 \\
      c g_1^2 h_1^2 l_1^2 & \text{leading coeff.} & (q-1)^2 & (q+1)^2 \\
      c g_1^2 h_2^2 & \text{leading coeff.} & q^2-1 & q^2-1 \\
      c g_3^2 & \text{leading coeff.} & q^2+q+1 & q^2-q+1 \\
      g_1^3 \ell_1^2 h_1 & \Res(\ell_1, g_1 h_1)
                             & q(q-1) & q(q+1) \\
      \ell_1^4 h_2 & \Res(\ell_1, h_2) & q(q-1) & q(q+1) \\
      c g_1^3 h_1^3 & - & q^2 & \\
      c g_2^3 & - & q^2 & \\
      c g_1^2 h_1^4 & \text{leading coeff.} & q(q-1) & q(q+1) \\
      g_1^5 h_1 & - & q^2 & \\
      c g_1^6 & \text{leading coeff.} & q^2 & q^2 \\\hline
   \end{array}
\]
\caption{Group orders $\#J^0_F(\F_q)$.} \label{table}
\end{figure}

\medskip

As in the proof of Thm.~\ref{Thm:J0GRoup},
we now consider the situation that $\CO$ is a complete discrete valuation
ring with uniformizer~$\pi$, residue field~$k$ such that $\Char(k) \neq 2$
and field of fractions~$L$. We denote by $v : L^\times \to \Z$ the normalized
valuation. Let $F \in \CO[X,Z]$ be homogeneous of degree~6
and squarefree. The 72 quadrics defining $J_F$ have coefficients in~$\CO$; we obtain a
flat scheme over~$\Spec(\CO)$. We abuse notation slightly and set
\[ J^0_F(L) = \{ P \in J_F(L) : \bar{P} \in J^0_{\bar F}(k) \} \quad\text{and}\quad
   J^1_F(L) = \{ P \in J_F(L) : \bar{P} = \bar{O} \} \,.
\]
We will call $J^1_F(L)$ the {\em kernel of reduction}. The reader should be
warned that this notion depends on the given model of the curve and need
not coincide with the kernel of reduction defined in terms of a N\'eron
model of the Jacobian.

\begin{Lemma} \label{LemmaBad1}
  Consider $(A, B) \in D_F(L)$ with
  $A(X,Z) = X^2 + a_1 X Z + a_0 Z^2$ and $B(X,Z) = b_1 X Z^2 + b_0 Z^3$.
  \begin{enumerate}\addtolength{\itemsep}{1mm}
    \item If $a_0, a_1 \in \CO$, but $b_0$ and $b_1$ are not both integral, 
          then $P = \phi(A,B)$ is in the kernel of reduction.
    \item Now assume that $a_0$, $a_1$, $b_0$ and~$b_1$ are integral.
          If $\pi$ divides $f_0$, $f_1$ and~$a_0$, but $\pi^2$ does not 
          divide~$f_0$, then $\pi$ also divides $a_1$
          and~$b_0$, but does not divide $f_2 - b_1^2$.
  \end{enumerate}
\end{Lemma}

\begin{proof} \strut
  \begin{enumerate}
    \item We work in the affine chart $(X : Z) = (x : 1)$.
          Reducing $F(x,1)$ modulo $A(x,1) = x^2 + a_1 x + a_0$, we obtain a
          relation $y^2 = \alpha_1 x + \alpha_0$ that holds for the points
          in the divisor described by the pair of polynomials $(A, B)$.
          Since the coefficients of~$F$ and $a_0, a_1$ are integral,
          the same holds for $\alpha_0$ and~$\alpha_1$.
          If we square the relation $y = B(x,1)$ and reduce it mod~$A(x,1)$,
          we obtain
          \[ b_1(2b_0 - b_1 a_1) = \alpha_1\,, \qquad
             b_0^2 - b_1^2 a_0 = \alpha_0 \,.
          \]
          The second relation shows that $v(b_0) < v(b_1)$ is impossible, 
          so we must have
          $v(b_1) < 0$. Eliminating~$b_0$ from the two equations above gives
          $(a_1^2 - 4 a_0) b_1^2 \in \CO$, so the discriminant of~$A(x,1)$
          must be divisible by~$\pi^2$. Therefore the two points in the divisor
          reduce mod~$\pi$ to points with the same $x$-coordinate. If these
          points were not opposite, then $y = B(x,1)$ would reduce to the
          equation of the (non-vertical) tangent line at the point on~$C_F(k)$
          that both points reduce to, so $b_0$ and~$b_1$ would be integral,
          contradicting the assumptions. So the divisor reduces to the sum
          of two opposite points, hence $P$ reduces mod~$\pi$ to the origin.
    \item We know that $x^2 + a_1 x + a_0$ divides $F(x,1) - (b_1 x + b_0)^2$.
          Write $f_0 = \pi f'_0$, $f_1 = \pi f'_1$, $a_0 = \pi a'_0$. From
          \begin{align*}
             f_6 x^6 &+ \dots + f_2 x^2 + \pi f'_1 x + \pi f'_0
              - (b_1 x + b_0)^2 \\
              &= (x^2 + a_1 x + \pi a'_0)
                 (c_4 x^4 + c_3 x^3 + c_2 x^2 + c_1 x + c_0) \,,
          \end{align*}
          we get
          \begin{align*}
             \pi f'_0 - b_0^2 &= \pi a'_0 c_0 \\
             \pi f'_1 - 2 b_0 b_1 &= a_1 c_0 + \pi a'_0 c_1 \\
             f_2 - b_1^2 &= c_0 + a_1 c_1 + \pi a'_0 c_2 
          \end{align*}
          The first of these implies that $b_0 = \pi b'_0$ for some 
          $b'_0 \in \CO$. Since $f'_0$ is not divisible by~$\pi$ (by
          assumption), we then also see that
          $\pi \nmid a'_0 c_0$. The second equation then shows that $\pi$
          divides~$a_1$, and then the third equation tells us that
          $f_2 - b_1^2 \equiv c_0 \not\equiv 0 \bmod \pi$.
  \end{enumerate}
\end{proof}

This allows us to get a description of the reductions of points not in
the kernel of reduction when the curve is regular.

\begin{Corollary} \label{CorBad2}
  Assume that $C_F/\CO$ as above is regular. Let $P \in J_F(L) \setminus J^1_F(L)$.
  If $P = \phi(A, B)$
  with $A(X,Z) \in \CO[X,Z]$ primitive, then after adding a suitable multiple 
  of~$A(X,Z)$, $B(X,Z)$ has coefficients in~$\CO$, and
  the reduction $(\bar{A}, \bar{B})$ of $(A, B)$ mod~$\pi$ is
  in~$D'_{\bar F}(k)$, hence $\bar{P}$ is a smooth point on~$J_{\bar F}$.
  In particular, if $\bar{F}$ is not a square, then $J_F(L) = J^0_F(L)$.
\end{Corollary}

\begin{proof}
  First assume that the coefficient of~$X^2$ in~$A(X,Z)$ is a unit. Then we
  can take $A(x, 1)$ to be monic and $B(x, 1)$ to be of
  degree at most~$1$. The integrality of~$B$ is given by 
  Lemma~\ref{LemmaBad1},~(1). If $\bar{A}$ vanishes at a singularity
  of~$C_{\bar F}$, then by a suitable shift, we can assume that
  the singularity is at $x = 0$ (we may have to extend the field for that;
  note that the shift will be by an integral element). We then have that
  $\pi$ divides $f_0$, $f_1$ and~$a_0$, which by Lemma~\ref{LemmaBad1},~(2),
  implies that $\pi$ also divides $a_1$ and~$b_0$ ($\pi^2 \nmid f_0$ because
  of the regularity assumption). This shows that $\bar{A}$ has a
  double root at the singularity (and hence, that no field extension was
  necessary) and that $\bar{B} = \lambda X Z^2$. We also know from the lemma that
  $\lambda^2 \neq \bar{f}_2$, which means exactly that the slope of the line 
  described by~$\bar{b}$ does not coincide with the slope of a branch
  of the curve at the singularity. Hence $(\bar{A}, \bar{B}) \in D'_{\bar F}(k)$.
  This implies that $\bar{P} \in J'_{\bar F}$.
  If $\bar{F}$ is not a square, then $J'_{\bar F}(k) = J^0_{\bar F}(k)$,
  and the last claim follows.
  
  The case when the coefficient of~$X^2$ in~$A(X,Z)$ is not a unit can be
  reduced to the general case discussed above by a suitable change
  of coordinates.
\end{proof}

\begin{Corollary} \label{Cor:RegQu}
  Assume that $C_F/\CO$ is regular and that $\bar{F}$ is not a square.
  Then the following sequence is exact.
  \[ 0 \To J^1_F(L) \To J_F(L)
      \stackrel{P \mapsto \bar P}{\To} J^0_{\bar F}(k) \To 0 \,.
  \]
\end{Corollary}

\begin{proof}
  By Cor.~\ref{CorBad2}, we know that $J^0_F(L) = J_F(L)$, and by
  the proof of Thm.~\ref{Thm:J0GRoup}, we know that reduction mod~$\pi$
  gives a group homomorphism $J^0_F(L) \to J^0_{\bar F}(k)$ with kernel
  $J^1_F(L)$. This homomorphism is surjective because of Hensel's Lemma
  (recall that the points in~$J^0_{\bar F}(k)$ are smooth).
\end{proof}

\begin{Remark}
  When $C_F/\CO$ is regular, then the scheme obtained from $J_F/\CO$ by
  removing the singular points in the special fiber~$J_{\bar F}/k$ is
  the N\'eron model of $J_F/L$, and $J^0_{\bar F}/k$ is the connected
  component of the identity on the special fiber.
  
  If $C_F/\CO$ is not regular, then the smooth part of $J_F/\CO$
  still maps to the N\'eron model (by the universal property of the latter),
  but the image of $J^0_{\bar F}$ in the special fiber of the N\'eron model
  can be trivial or a one-dimensional subgroup.
\end{Remark}

\medskip

We now consider a genus~2 curve~$C = C_F$ over~$\Q_p$ given by a Weierstrass equation
$Y^2 = F(X,Z)$ over~$\Z_p$. We will drop the subscript~$F$ in the following.
By the above, we have 
$J^0(\Q_p)/J^1(\Q_p) \cong J^0_{\bar F}(\F_p)$, and the map is given by reducing
the standard representation modulo~$p$ (on elements that are not in
the kernel of reduction).

This gives us a handle on the quotient $J(\Q_p)/J^1(\Q_p)$ when $p$ is
odd, the model is regular and the special fiber of~$C$ has just one
component, cf.\ Cor.~\ref{Cor:RegQu}.

Since we have now established that we can use Cantor reduction on 
$J^0_{\bar F}(\F_p)$ in the same way as in the good reduction case, we can proceed 
and find the image $\iota\bigl(C(\F_p)\bigr) \subset J^0_{\bar F}(\F_p)$
in the same way as described in Section~\ref{SubS:S}.

\smallskip

Otherwise, that is, when $p = 2$, the model is not regular, or the special fiber
has several components,
we first need to find $J^0(\Q_p)$, or rather (for our purposes)
$J(\Q) \cap J^0(\Q_p)$. We can do this by an enumerative process.

In the following, $A$ is a finitely generated free abelian group,
$t$ is a test that determines whether a given element of~$A$ is in the subgroup.
In our application, $A = J(\Q)$, and $t$ tests whether a point~$P$ is in~$J^0(\Q_p)$.
According to Prop.~\ref{Prop:J0K0}, we can use
\[ t(P) \iff v_p(\delta(\kappa(P))) = 4 v_p(\kappa(P)) \]
(with the same choice of projective coordinates for $\kappa(P)$ on both sides),
or in the notation of~\cite{StollH2}, $t(P) \iff \epsilon_p(P) = 0$.

\pagebreak[1]
\strut\hrulefill\\[3pt]
{\sf GetSubgroup}($A$, $t$): \\
\strut\quad $g$ := $\emptyset$ \qquad
  // $g$ will contain the generators of the subgroup \\
\strut\quad $A'$ := $\{0\} \subset A$ \quad // known part of quotient group \\
\strut\quad {\sf for} $b \in \text{Generators($A$)}$: \\
\strut\qquad // find smallest multiple of $b$ such that $A' + b$ meets the subgroup \\
\strut\qquad $j$ := 1; \quad $b'$ := $b$ \\
\strut\qquad {\sf while} $\neg \exists a \in A': t(b'+a)$: \\
\strut\qquad\quad $j$ := $j+1$; \quad $b'$ := $b' + b$ \\
\strut\qquad {\sf end while} \\
\strut\qquad // note new subgroup generator \\
\strut\qquad $g$ := $g \cup \{b' + a\}$, where $a \in A'$ satisfies $t(b'+a)$ \\
\strut\qquad // extend $A'$ to get a set of representatives of the image of the group \\
\strut\qquad // generated by the first few generators of $A$ in the quotient \\
\strut\qquad $A'$ := $\{ a + i \cdot b : i \in \{0, \dots, j-1\}, a \in A' \}$ \\
\strut\quad {\sf end for} \\
\strut\quad {\sf return} $\langle g \rangle$ // a subgroup of $A$ \\[-3pt]
\strut\hrulefill

\smallskip

This allows us to find $J(\Q) \cap J^0(\Q_p)$ and hence the image of~$J(\Q)$
in $J(\Q_p)/J^0(\Q_p)$. It remains to determine the image of~$C(\Q_p)$ in 
this group. It is, however, better to find the image of~$C(\Q_p)$ 
in~$J(\Q_p)/J^1(\Q_p)$ directly, or rather, to find the subset of
$J(\Q)/\bigl(J(\Q) \cap J^1(\Q_p)\bigr)$ that is in the image of~$C(\Q_p)$.
For this we use the map to the dual Kummer surface described below in 
Section~\ref{S:deep}: for a representative $P \in J(\Q)$ of each element
of $J(\Q)/\bigl(J(\Q) \cap J^1(\Q_p)\bigr)$, we check if its image on the
dual Kummer surface satisfies $p^2 \mid \eta_4$ and $p \mid \eta_1 \eta_3 - \eta_2^2$.

The reason for working mod~$J^1(\Q_p)$ and not mod~$J^0(\Q_p)$ (which might
be more efficient) is that there does not seem to be a simple criterion
that tells us whether we are in $\iota\bigl(C(\Q_p)\bigr) + J^0(\Q_p)$.


\section{`Deep' information} \label{S:deep}

In this section, we work with genus~2 curves over~$\Q$ for simplicity.
Everything can easily be generalized to genus~2 curves over arbitrary
number fields.

Especially for small primes~$p$, we can hope to gain valuable information
by not just looking at~$J(\F_p)$ or, more generally, $J(\Q_p)/J^1(\Q_p)$,
but also into the kernel of reduction to some depth. If $J^n(\Q_p)$
(for $n \ge 1$) denotes the `$n$th kernel of reduction', i.e., the
subgroup of elements that consists of the $p^n \Z_p$-points of the formal group,
then we would like to determine (the image of~$J(\Q)$ in) 
$J(\Q_p)/J^n(\Q_p)$ and the image of~$C(\Q_p)$ in this group.

The first step is to find $J(\Q) \cap J^n(\Q_p)$. This can be done 
with the help of the $p$-adic logarithm
on the Jacobian. The power series of the formal logarithm up to terms
of degree~7 can be found on Victor Flynn's website~\cite[local/log]{FlynnFTP}.
If higher precision is needed, we perform a $p$-adic numerical integration,
as follows. We can represent a given point in the kernel of reduction
in the form $[P_1 - P_2]$, where $P_1$ and~$P_2$ are points on the curve
that reduce mod~$p$ to the same point. Assuming for simplicity that the
points have $p$-adically integral coordinates and do not reduce to a
Weierstrass point, we write 
$P_1 = (\xi + \delta, \eta_1)$, $P_2 = (\xi - \delta, \eta_2)$.
We then write the differentials $\omega_0 = dx/2y$ and $\omega_1 = x\,dx/2y$
as a power series in terms of the uniformizer $t = x - \xi$, times~$dt$,
and integrate this numerically from $t = -\delta$ up to $t = \delta$,
to the desired precision (note that $\delta$ has positive valuation).
Alternatively, we can use that on the Kummer surface, we have
\[ p^{n-1} \cdot P = \bigl(\lambda_1^2 \, p^{2n} + O(p^{3n})
                           : 2 \lambda_1 \lambda_2 \, p^{2n} + O(p^{3n})
                           : \lambda_2^2 \, p^{2n} + O(p^{3n}) : 1\bigr)
\]
where $(\lambda_1, \lambda_2)$ is the logarithm of~$P$. So to compute
the logarithm up to~$O(p^n)$, we multiply the point by~$p^{n-1}$ on the
Kummer surface to find the logarithm up to a sign. (If $p = 2$, we
need a few more bits of precision here.) We then fix the sign
by comparing with the first-order approximation we obtain from the 
functions $\lambda$ and~$\mu$ on the Jacobian, in the notation 
of~\cite[\S\,2]{CasselsFlynn}.

Given that we are able to compute the logarithm
\[ \log : J^1(\Q_p) \To (p \Z_p)^2 \]
to any desired accuracy, we compute the finite-index subgroup
$K_n = J(\Q) \cap J^n(\Q_p)$ of~$J(\Q)$ as follows. We assume
that $K_1$ is already given. We can therefore set up the group
homomorphism
\[ K_1 \stackrel{\log}{\To} (p \Z_p)^2
       \To \Bigl(\frac{p \Z_p}{p^n \Z_p}\Bigr)^2
       \cong \Bigl(\frac{\Z}{p^{n-1}\Z}\Bigr)^2 \,;
\]
then $K_n$ is just its kernel.

\medskip

The second and more time-consuming step is to find the image of
$C(\Q_p)$ in $J(\Q_p)/J^n(\Q_p)$. We assume again that the `flat'
information (i.e., the image of $C(\Q_p)$ in $J(\Q_p)/J^1(\Q_p)$)
is already known. For each point in the intersection of the images
of~$C(\Q_p)$ and of~$J(\Q)$ inside $J(\Q_p)/J^1(\Q_p)$, we then
have to find all its `liftings' to elements in the intersection
of the images of~$C(\Q_p)$ and of~$J(\Q)$ in $J(\Q_p)/J^n(\Q_p)$.

One approach would be to take some lifting $P_0$ in~$C(\Q_p)$, add
representatives of $J(\Q_p)/J^n(\Q_p)$ to it and see which lie
sufficiently close to~$C$. One practical problem lies in the word `add'.
By~\cite[Ch.~2,3]{CasselsFlynn}, the Jacobian can be embedded
into~$\BP^{15}$, and the sum $P+Q$ can be expressed in terms of biquadratic
forms in the coordinates of $P$ and~$Q$. For a given curve~$C$, these
forms can be determined using interpolation, but they can have several
thousand terms, so any subsequent computations based on them will be
rather slow.

The usual method of adding points on~$J$, following~\cite{Cantor},
essentially uses some affine part of the Jacobian. Problems with
denominators make it not well-suited for $p$-adic fixed precision
calculations.

We instead propose to use the Kummer surface and its dual
(see~\cite[Ch.~4]{CasselsFlynn}).
The hyperelliptic involution on~$C$ induces an involution on the
principal homogeneous space $\Pic^1_C$ of~$J$, and the quotient
of~$\Pic^1_C$ by this involution is again a quartic surface in~$\BP^3$.
An explicit equation is given by
\[ \psi(\eta_1, \eta_2, \eta_3, \eta_4)
    := \left|\begin{matrix}
              2 f_0 \eta_4 & f_1 \eta_4 & \eta_1 & \eta_2 \\
              f_1 \eta_4 & 2 f_2 \eta_4 - 2 \eta_1 & f_3 \eta_4 - \eta_2 & \eta_3 \\
              \eta_1 & f_3 \eta_4 - \eta_2 & 2 f_4 \eta_4 - 2 \eta_3 & f_5 \eta_4 \\
              \eta_2 & \eta_3 & f_5 \eta_4 & 2 f_6 \eta_4
             \end{matrix}\right| = 0\,,
\]
see~\cite[p.~33]{CasselsFlynn}. This model has the property that the
natural image of~$C \subset \Pic^1_C$ is given
by~$\eta_4 = 0$. Furthermore, if $P \in \Pic^1_C$
maps to $(\eta_1 : \eta_2 : \eta_3 : \eta_4)$ on the dual Kummer surface
and $Q \in J$ maps to 
$(\xi_1 : \xi_2 : \xi_3 : \xi_4)$ on the Kummer surface,
then $P \in C \pm Q$ if and only if 
$\xi_1 \eta_1 + \xi_2 \eta_2 + \xi_3 \eta_3 + \xi_4 \eta_4 = 0$.
We will denote the Kummer surface by~$\CK$ and the dual Kummer surface
by~$\CK^*$.

The group law on $J$ leaves its traces on~$\CK$. Suppose that $Q,R \in J$.
Write $\mathbf{y}=(\xi_1(Q),\ldots,\xi_4(Q))$ and 
$\mathbf{z}=(\xi_1(R),\ldots,\xi_4(R))$ for projective coordinates of
their images on~$\CK$. Following~\cite[Ch.~3]{CasselsFlynn}, there is
a matrix of biquadratic forms $B(\mathbf{y},\mathbf{z})=(B_{ij})$ such that
\[ 2B_{ij} = \xi_i(Q+R) \xi_j(Q-R) + \xi_i(Q-R) \xi_j(Q+R) \,. \]
The action of $J$ on~$\Pic^1_C$ can be similarly described on~$\CK^*$.
Suppose that $Q \in J$ and $P \in \Pic^1_C$ and that $\mathbf{x} = \eta(P)$ and 
$\mathbf{y} = \xi(Q)$ are projective coordinates for their images on~$\CK^*$
and on~$\CK$, respectively. There is now a symmetric matrix of biquadratic forms 
$A(\mathbf{x},\mathbf{y}) = (A_{ij})$ such that
\[ 2A_{ij} = \eta_i(P+Q) \eta_j(P-Q) + \eta_i(P-Q) \eta_j(P+Q) \,. \]
The following result lets us compute $A$ from~$B$ rather easily.
We assume that $B$ has been scaled so that $B_{44}(0,0,0,1; 0,0,0,1) = 1$
and $A$ has been scaled so that $A_{11}(1,0,0,0; 0,0,0,1) = 1$.

\begin{Lemma}
  Let $\mathbf{x}$ be coordinates of the image of~$P \in \Pic^1_C$ on~$\CK^*$,
  and let $\mathbf{y}$, $\mathbf{z}$ be coordinates of the
  images of $Q, R \in J$ on~$\CK$. Then, considering
  $\mathbf{x}$, $\mathbf{y}$, $\mathbf{z}$ as row vectors,
  \[ \mathbf{x} \, B(\mathbf{y}, \mathbf{z}) \, \mathbf{x}^\top
      = \mathbf{z} \, A(\mathbf{x}, \mathbf{y}) \, \mathbf{z}^\top \,.
  \]
\end{Lemma}

\begin{proof}
  Both sides are triquadratic forms in $\mathbf{x}, \mathbf{y}, \mathbf{z}$.
  Using the duality property mentioned above, it can be checked that
  each side vanishes if and only if 
  \[ P \in C \pm Q \pm R \text{\quad for some choice of signs.} \] 
  This implies that both sides are proportional, and
  since they take the same value~$1$ at $\mathbf{x} = (1,0,0,0)$,
  $\mathbf{y} = \mathbf{z} = (0,0,0,1)$, they must be equal (since there
  are no quadrics vanishing on either of the two surfaces).
\end{proof}

So in order to find~$A$, we construct the polynomial on the left hand
side and interpret it as a quadratic form in~$\mathbf{z}$.

On the Kummer surface, we can use~$B$ to find the image of $P+Q$ if
the images of $P$, $Q$ and~$P-Q$ are known. This is known as `pseudo-addition'
(see~\cite{FlynnSmart}) and can be extended to the 
computation of images of linear combinations $a_1 P_1 + \dots + a_m P_m$ if
the images of the $2^m$ points $e_1 P_1 + \dots + e_m P_m$ are known, where 
$e_j \in \{0, 1\}$. It should be noted that the complexity of this
procedure in terms of pseudo-additions is~$2^m$ times the bit-length
of the coefficients, so we should not use it to compute linear combinations
of many points. One important feature of this is that it works with
projective coordinates and is therefore well-suited for $p$-adic arithmetic
with fixed precision.

In a similar way, we can compute the image of 
$P + a_1 P_1 + \dots + a_m P_m$ on the dual Kummer surface, if $P \in \Pic^1_C$
and $P_1, \dots, P_m \in J$. We need to know the images of
$P + e_1 P_1 + \dots + e_m P_m$ (with $e_j \in \{0,1\}$) in addition
to $e_1 P_1 + \dots + e_m P_m$, and in the
pseudo-addition step, $B$ is replaced with~$A$. The remark on complexity
applies here as well. Below, we will take generators of the successive
quotients $K_{l-1}/K_l$ as the~$P_j$; in most cases, this quotient is
isomorphic to a subgroup of $(\Z/p\Z)^2$, so that $m \le 2$.

The following lemma tells us how to find the subset of $J(\Q_p)/J^n(\Q_p)$
of elements such that the corresponding cosets of~$J^n(\Q_p)$ meet the
image of the curve.

\begin{Lemma}
  Let $P_0 \in C(\Q_p)$, and let $Q \in J^n(\Q_p)$. If we normalize the coordinates
  $(\eta_1 : \eta_2 : \eta_3 : \eta_4)$ of the image of $P_0+Q$ on~$\CK^*$
  so that the minimal $p$-adic valuation is zero, then
  \[ v_p(\eta_1 \eta_3 - \eta_2^2) \ge n \quad\text{and}\quad
     v_p(\eta_4) \ge 2n \,.
  \]
\end{Lemma}

\begin{proof}
  Let $P$ be the image on~$\CK^*$ of $P_0 \in C(\Q_p)$.
  If we make an invertible coordinate change over~$\Z_p$ on the $\BP^1$
  that $C$ maps to, then this induces an invertible coordinate change over~$\Z_p$
  on the ambient projective spaces of~$\CK$ and of~$\CK^*$,
  which leaves
  the valuations of $\eta_1 \eta_3 - \eta_2^2$ and of~$\eta_4$ invariant. We
  can therefore assume without loss of generality that the point on the curve
  is at infinity. Then $P = (1 : 0 : 0 : 0)$.
  
  Since $Q \in J^n(\Q_p)$, its image on~$\CK$ has coordinates
  of the form
  \[ (\alpha p^{2n} : \beta p^{2n} : \gamma p^{2n} : 1)
       \text{\quad with $\alpha, \beta, \gamma \in \Z_p$.}
  \]
  Denote the coordinates of the images of $P_0 \pm Q$ on~$\CK^*$
  by $(\eta_1 : \eta_2 : \eta_3 : \eta_4)$
  and~$(\eta'_1 : \eta'_2 : \eta'_3 : \eta'_4)$.
  If we evaluate the entries of the matrix~$A$ at the coordinates of $P$ and~$Q$,
  then by the definition of~$A$ we have (with suitable scaling)
  \[ 2 A(P,Q) = (\eta_1, \eta_2, \eta_3, \eta_4)^\top
                   \,(\eta'_1, \eta'_2, \eta'_3, \eta'_4)
                + (\eta'_1, \eta'_2, \eta'_3, \eta'_4)^\top
                    \,(\eta_1, \eta_2, \eta_3, \eta_4) \,.
  \]
  We obtain 
  \[ \eta_1 \eta'_1 = A(P,Q)_{11} \equiv 1 \bmod p^{2n} \,, \] 
  so that we can scale the coordinates to have 
  \[ \eta_1 \equiv \eta'_1 \equiv 1 \bmod p^{2n} \,. \]
  We then find that
  \[ \eta_4 + \eta'_4 \equiv 2A(P,Q)_{14} \equiv 0 \bmod p^{2n} \quad\text{and}\quad
     \eta_4 \eta'_4 = A(P,Q)_{44} \equiv 0 \bmod p^{4n}\,;
  \]
  this implies that
  \[ \eta_4 \equiv \eta'_4 \equiv 0 \bmod p^{2n} \,. \]
  
  All entries in $A(P,Q)$, except~$A_{11}$, have valuation at least~$2n$.
  It follows in a similar way as above that 
  \[ \eta_2, \eta'_2, \eta_3, \eta'_3 \equiv 0 \bmod p^n \]
  and therefore that
  \[ \eta_1 \eta_3 - \eta_2^2 \equiv \eta'_1 \eta'_3 - {\eta'_2}^2
     \equiv 0 \bmod p^n
  \]
  as claimed.
\end{proof}

Recall that we have fixed an embedding $\iota : C \to J$, given by some
rational divisor (class) of degree~1 on~$C$. This induces an isomorphism
$\iota : \Pic^1_C \stackrel{\simeq}{\to} J$.
So in order to test whether an element of~$J(\Q)/K_n$ is in the image
of~$C(\Q_p)$, we map a representative in~$J(\Q)$ to~$\Pic^1_C$ via~$\iota^{-1}$
and then to the dual Kummer surface, and check whether the normalized coordinates 
of the image satisfy 
\[ v_p(\eta_4) \ge 2n \quad\text{and}\quad v_p(\eta_1 \eta_3 - \eta_2^2) \ge n \,. \]
Note that we can compute the image on the dual Kummer surface if we
know the images of $e_1 P_1 + \dots + e_m P_m$ on~$\CK$
and on~$\CK^*$, where the~$P_j$ are
representatives of generators of~$J(\Q)/K_n$ (with $e_j \in \{0, 1\}$).

If we proceed as just described, then we need to enumerate $J(\Q)/K_n$
(of size approximately~$p^{2n}$) in order to find the image of~$C$, which
is of size approximately~$p^n$. We can make several improvements in order
to reduce the complexity to something closer to the lower bound of~$O(p^n)$.
One improvement is to
compute the images successively for $n = 2, 3, \dots$. When we go from
$n = m$ to $n = m+1$, we only have to consider group elements that map
into the image of the curve on the previous level; there will usually be
$p^2$ of these for each of the roughly $p^m$ elements in the previous image.
This gives a complexity of $p^{m+2}$ for this step, and a total complexity
of $\frac{p^2}{p-1}\,p^n$. This is still worse by a factor of $p^2/(p-1)$
than what we would get if we could compute the images of points in~$C(\Q_p)$
in $J(\Q_p)/J^n(\Q_p)$ directly, but it is reasonably good for applications.

We can further improve on this in many cases. Let $P \in J(\Q)$ such that its
image on~$\CK^*$ satisfies $v_p(\eta_1 \eta_3 - \eta_2^2) \ge m$ as above.
We work in an affine patch of~$\CK^*$ such that the image of~$P$ has
$p$-adically integral coordinates and write $h(P)$ for the function
$\eta_1 \eta_3 - \eta_2^2$, evaluated at~$P$ in terms of these affine
coordinates. The theory of formal groups implies that the map
\[ J^m(\Q_p) \To \F_p \,, \quad
   Q \longmapsto p^{-m} \bigl(h(P+Q) - h(P)\bigr) \bmod p
\]
is linear, with kernel containing $J^{m+1}(\Q_p)$. This gives us a linear
form $\ell_m : K_m/K_{m+1} \to \F_p$. If $\ell_m$ is nonzero, then we only
need to evaluate it on a generating set of~$K_m/K_{m+1}$ in order to
find the points $Q \in K_m$ such that $v_p(h(P+Q)) \ge m+1$. Since
$K_m/K_{m+1}$ usually has two generators, this gives a complexity of
order $(2+p) p^m$: for each of the roughly $p^m$ points~$P$, we have
to evaluate $\ell_m$ on the two generators and then compute the (usually)
$p$ lifts to the next level. Note that the linear form is nonzero
on $J^m(\Q_p)/J^{m+1}(\Q_p)$ if and only if the reduction mod~$p$ of the
image of~$P$ on~$\CK^*$ is nonsingular. This is the case unless
$p = 2$ or the corresponding point in~$C(\F_p)$ has vanishing $y$-coordinate.
So if $p$ is an odd prime such that the polynomial defining~$C$ is not
divisible by~$p$, there will be at most six `problematic' classes mod~$p$,
contributing at most $6p^2$ to the complexity at each step. The overall
complexity is therefore~$O(p^n)$ for such primes, which is of the order
of the obvious lower bound.


\section{Implementation} \label{S:Imp}

In this section, we describe a concrete implementation of the Mordell-Weil
sieve on genus~2 curves that can be used to prove that a given curve does
not have a rational point. For this implementation, the {\sf MAGMA}
computer algebra system~\cite{Magma} was used.
Our implementation is available at~\cite{Software}.

We assume that we are given as input
\begin{enumerate}
  \item the polynomial $f(x)$ on the right hand side of the equation
        $y^2 = f(x)$ of the curve~$C$,
  \item generators of the Mordell-Weil group $J(\Q)$, where $J$ is the
        Jacobian variety of the curve, and
  \item a rational divisor~$D$ of degree~$3$ on the curve.
\end{enumerate}
The latter is used to provide the embedding $\iota : C \to J$, which
is given by sending a point~$P \in C$ to the class of $P + W - D$,
where $W$ is a canonical divisor.

Elements of~$J(\Q)$ can be represented by divisors of degree~$2$, and
divisors can be represented by pairs $(a, b)$ of polynomials as
in Section~\ref{S:bad} above. We let $r$ denote the rank of~$J(\Q)$.

In the first step, we have to provide the necessary input for the actual
sieving procedure. This means that we have to determine the group structure
of~$J(\F_p)$, the reduction homomorphism $\phi_p : J(\Q) \to J(\F_p)$, and the
image of $C(\F_p)$ in~$J(\F_p)$ in terms of this group structure. This
involves the computation of $r + t + \#C(\F_p)$ discrete logarithms in
the group~$J(\F_p)$, where $r$ is the Mordell-Weil rank and $t$ is the number
of generators of the torsion subgroup of~$J(\Q)$. The first $r + t$ of these
are needed to find~$\phi_p$, and the others are needed to find the image
of~$C(\F_p)$ in~$J(\F_p)$, represented by the abstract group~$G_p$.
If we restrict to primes~$p$ such that $\#J(\F_p)$ is $B$-smooth, then
we can use Pohlig-Hellman reduction~\cite{PohligHellman}
for the computation of the discrete
logarithms, so that the complexity of this step is about $r + t + \#C(\F_p)$
(assuming $B$ is fixed). The total effort required for the computation
in the first step is therefore
\[ {} \approx (r + t) \#S + \sum_{p \in S} \#C(\F_p)
      \approx (r + t) \#S + \sum_{p \in S} p
      \approx \bigl(r + t + \tfrac{1}{2} \max S\bigr) \#S \,.
\]
In the last estimate, we have made the simplifying assumption that the primes
in~$S$ are distributed fairly regularly, so the factor~$\frac{1}{2}$ will
not be completely accurate. The point is that this is essentially quadratic
in $\#S$ or $\max S$. So the relevant question is how far we have to go
with $\max S$ in order to collect enough information to make success likely.

A reliable theoretical analysis of this question appears to be rather
difficult, although one could try to get some information out of an
approach along the lines of Poonen's heuristic~\cite{PoonenHeur}.
Therefore we use the following approach. We compute the relevant information
for each prime~$p$ (such that $\#J(\F_p)$ is $B$-smooth) in turn. Then we
compute the numbers $n(S, N_{l-1-r-j})$ for $j = 0,1,2,3$ in the notation
of Section~\ref{SubS:S}, where $S$ is the set of primes used so far.
This can be done incrementally, caching the values of
$\#C_{N,p}/\#(G_p/NG_p)$ for later use (they only depend on the gcd of~$N$
and the exponent of~$G_p$), and does not cost much time.
We stop this part of the computation when
\[ \min_j n(S, N_{l-1-r-j}) < \eps \]
for a given parameter $\eps \ll 1$. Tests performed with the `small curves'
from~\cite{BruinStollExp} indicate that $\eps = 0.01$ is a reasonable
choice and that $B = 200$ leads to good results.
Figure~\ref{fig:nS} shows the dependence of $n(S, N_{l-1-r-j})$
from~$\max S$ in a fairly typical example (of rank~3).

\begin{figure}[htb]
  \begin{center}
    \Gr{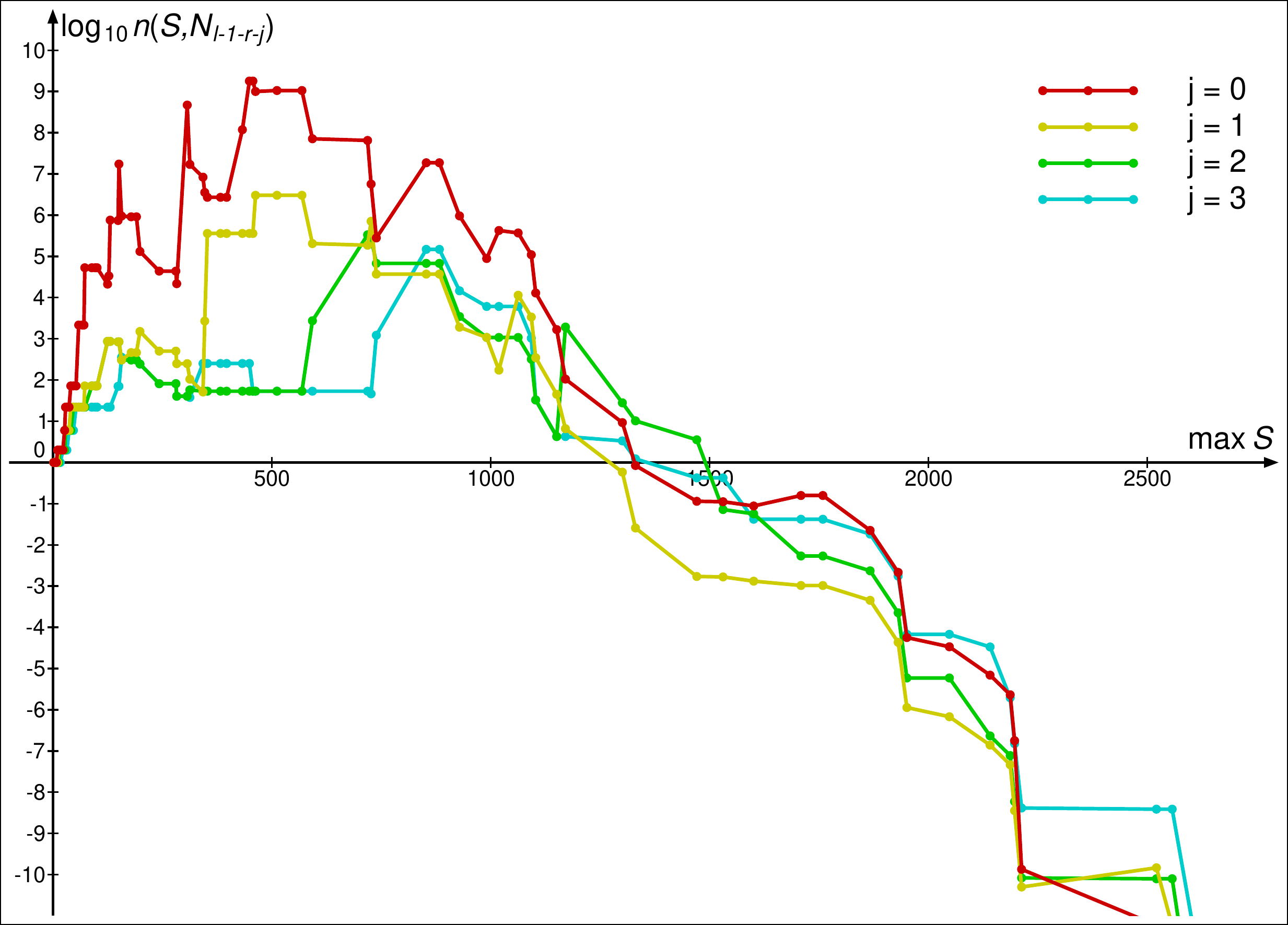}{0.8\textwidth}
    \caption{Expected sizes $n(S,N_{l-1-r-j})$ versus $\max S$} \label{fig:nS}
  \end{center}
\end{figure}

We include the computation of `bad' and `deep' information 
(as described in Sections \ref{S:bad} and~\ref{S:deep} above) as we go along.
We let $n = 2, 3, 4, \dots$, and when $n = p$ is a prime,
we compute information mod $p$ if $p < 10$, or $p \le B$ and the
given model of~$C$ is regular at~$p$ such that $C/\F_p$ has only one
component, or $p$ is a prime of good reduction and $\#J(\F_p)$ is $B$-smooth.
If $n = p^v$ is a prime power $p^v$,
then we compute information mod~$p^m$ with $m = (v+1)/2$ if $v$ is odd.
This scheme proved to give the best performance with our implementation.
It hits a good balance
between the effort required to compute the information (which is much
greater than for `flat' and `good' information at primes $q \approx p^m$)
and the gain in speed resulting from the additional information. The
information mod~$p^m$ is therefore computed in the following order.
\[ p^m = 2, 3, 5, 7, 2^2, 11, 13, 17, 19, 23, 3^2, 29, 31, 2^3, 37, \dots \]

After the information has been collected, we
compute a `$q$ sequence' as described in Section~\ref{SubS:N},
using a target value of~$\eps_1$ with
$\eps < \eps_1 < 1$. We take $\eps_1 = 0.1$ as the standard value of this
parameter. Since $\eps_1 > \eps$, we know from the first part of the
computation that a suitable sequence exists. If we take $\eps_1$ not too
close to~$\eps$, this second part of the computation is usually rather fast.

Finally, we use the collection $\{(G_p, \phi_p, C_p) : p \in S\}$ and
the $q$ sequence as input for the actual sieve computation.
This computation is done as described in Section~\ref{SubS:A}.
If it does not result in the desired
contradiction, we divide the $\eps$ and $\eps_1$ parameters by~$10$ and start over
(keeping the local information we have already computed).


\section{Efficiency} \label{S:Eff}

How long do our computations take? Let us look at the various
steps that have to be performed, in the context of the first application
discussed in Section~\ref{S:Appl}
above: verifying that a given curve~$C$ of genus~2 over~$\Q$ does
not have rational points. We assume that a Mordell-Weil basis is known.
Note that in practice, the part of the computation that determines
this Mordell-Weil basis can be rather
time-consuming, but this is a different problem, which we will not
consider here. See~\cite{Stoll2DH,StollH1,StollH2} for the relevant
algorithms. We also assume that
we know a rational divisor of degree~3 on~$C$. Again, it might be not
so easy to find such a divisor in practice.

We consider the 1447 curves for which we had to perform a Mordell-Weil
sieve computation in~\cite{BruinStollExp} in order to rule out the
existence of rational points. The difference to the 1492~curves mentioned
earlier comes from the fact that some curves had rank zero, and some
others could be ruled out immediately by the information coming from
the Birch and Swinnerton-Dyer conjecture.
The timings mentioned below were obtained on
a machine with 4~GB of RAM and a 2.0~GHz dual core processor.
As before, $r$ denotes the Mordell-Weil rank.

Among the 521 curves with $r = 1$, there are 514 such that we already
obtain a contradiction while collecting the information. This occurs
when we find a prime~$p$ or prime power $p^n$ such that the images of~$J(\Q)$ 
and of~$C(\Z/p^n\Z)$ in~$J(\Z/p^n\Z)$ are disjoint. 
It is perhaps worth noting that without looking at `bad' and `deep'
information, we obtain this kind of immediate contradiction only for 406~curves.
The average computing time for a single
curve was about 0.1~seconds, and the longest time was about 6.3~seconds.
The distribution of running times is shown in Figure~\ref{fig:rk1}
(on a logarithmic scale).

\begin{figure}[htb]
  \begin{center}
    \Gr{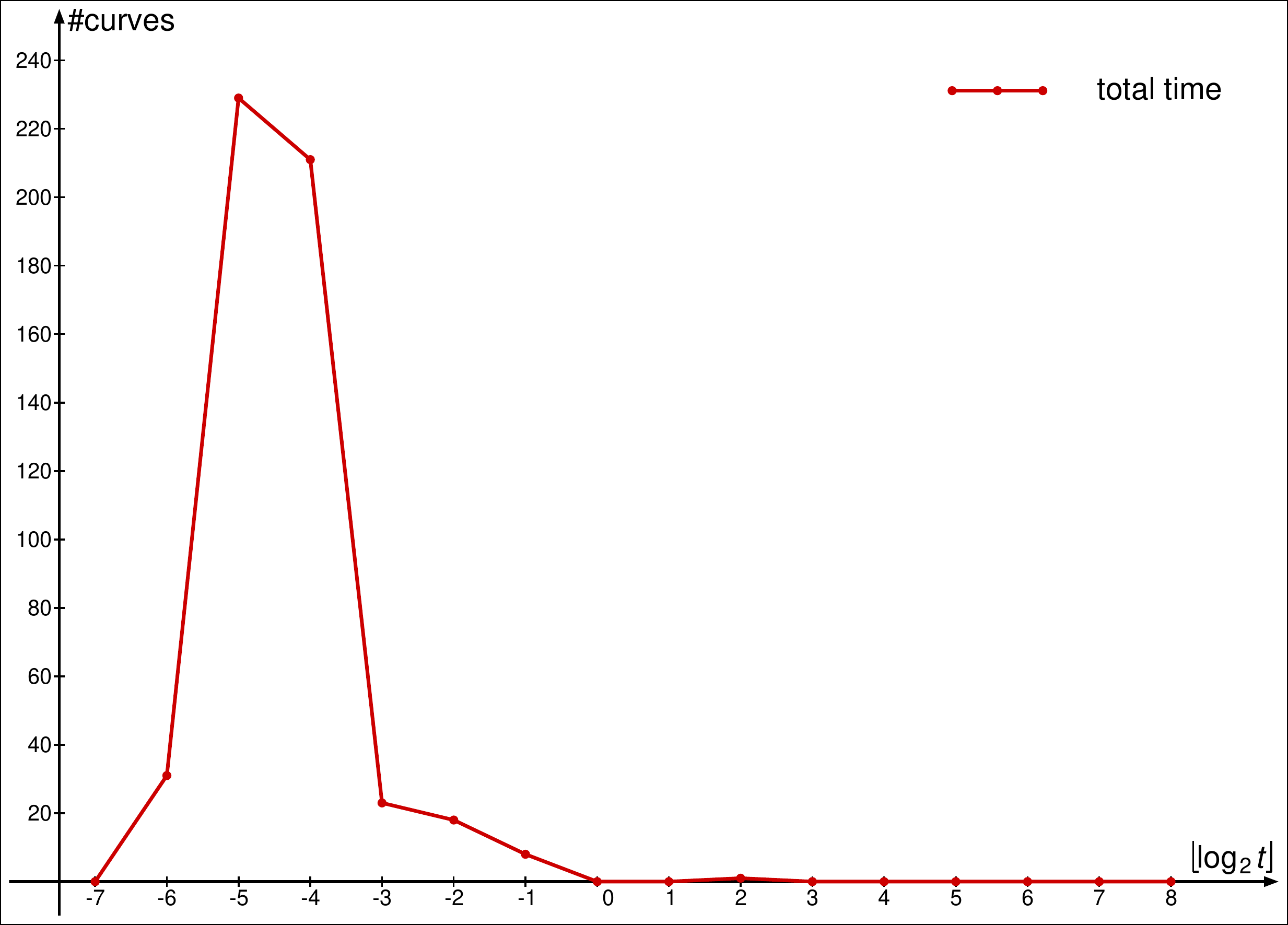}{0.8\textwidth}
    \caption{Running times for $r = 1$} \label{fig:rk1}
  \end{center}
\end{figure}

The anonymous referee asked whether there is a heuristic explanation
for the observation that information at one prime is almost always
enough to rule out rational points. Here is an attempt at such an
explanation. We use the following probabilistic model. We assume
that $J(\F_p)$ is cyclic of order uniformly distributed in an interval
around~$p^2$ of length $\asymp p^{3/2}$, that the generator~$P_0$ of~$J(\Q)$
(which we assume to be torsion-free of rank one) is mapped to a random
element of~$J(\F_p)$ and that the points in~$C(\F_p)$ form a random
subset of~$J(\F_p)$. We are interested in the probability that $C(\F_p)$
and the image of~$J(\Q)$ in~$J(\F_p)$ are disjoint. Note that the case
when $J(\F_p)$ is cyclic is the worst case; if $J(\F_p)$ is not cyclic,
then the cyclic image of~$J(\Q)$ will be more likely to be small.

\begin{Lemma}
  In the model described above, the probability that $C(\F_p)$ does not
  meet the image of~$J(\Q)$ is $\gg 1/p$.
\end{Lemma}

\begin{proof}
  Let $n = p^2 + O(p^{3/2})$ be the order of~$J(\F_p)$, denote the
  index of the image of~$J(\Q)$ in~$J(\F_p)$ by~$d$, and let
  $m = p + O(p^{1/2})$ denote~$\#C(\F_p)$. Then the conditional probability,
  given that the index is~$d \ge 2$, is
  \begin{align*}
    q_d &= \frac{\binom{n - n/d}{m}}{\binom{n}{m}}
         = \prod_{k=0}^{m-1} \Bigl(1-\frac{1}{d(1-k/n)}\Bigr) \\
        &= \exp\Bigl(\sum_{k=0}^{m-1} \log\Bigl(1-\frac{1}{d(1-k/n)}\Bigr)\Bigr)
         = \exp\Bigl(-\sum_{k=0}^{m-1} \frac{1}{d(1-k/n)} + O(p d^{-2})\Bigr) \\
        &= \exp\Bigl(-\frac{1}{d} \int_0^{m} \frac{dt}{1-t/n}
                        + O(d^{-1}) + O(p d^{-2})\Bigr) \\
        &= \exp\Bigl(\frac{n}{d}\,\log \Bigl(1 - \frac{m}{n}\Bigr)
            + O(d^{-1}) + O(p d^{-2})\Bigr) \\
        &= \exp\Bigl(-\frac{m}{d} + O(d^{-1}) + O(p d^{-2})\Bigr) \,.
  \end{align*}
  Here $O(d^{-1})$ denotes a quantity that is bounded by a constant times~$d^{-1}$,
  and $O(p d^{-2})$ denotes a quantity that is bounded by a constant
  times~$p d^{-2}$ for large~$p$.
  
  We restrict to the range $\alpha p \le d \le \beta p$ with fixed
  $0 < \alpha < \beta$. Then
  \[ q_d = e^{-m/d} \bigl(1 + O(p^{-1})\bigr)
         = e^{-p/d} \bigl(1 + O(p^{-1/2})\bigr) \,.
  \]
  We now have to estimate the probability that $d$ has a given value~$d_0$
  in the range under consideration. Fix a generator~$Q$ of~$J(\F_p)$ and write 
  $\bar{P}_0 = k \cdot Q$, where $\bar{P}_0$ is the image of~$P_0$ in~$J(\F_p)$.
  Then the probability is
  \begin{align*}
    \Prob(d = d_0) &= \frac{\#\{(n,k) : 0 \le k < n = p^2 + O(p^{3/2}),
                                        \gcd(n,k) = d_0\}}%
                           {\#\{(n,k) : 0 \le k < n = p^2 + O(p^{3/2})\}} \\
                   &= \frac{6}{\pi^2 d_0^2} \Bigl(1 + O(p^{-1/2+\varepsilon})\Bigr) \,.
  \end{align*}
  So the total probability can be bounded below by
  \begin{align*}
    \sum_{\alpha p \le d_0 \le \beta p} \Prob(d = d_0) q_{d_0}
      &= \frac{6}{\pi^2} \sum_{\alpha p \le d_0 \le \beta p}
           \frac{1}{d_0^2}\,e^{-p/d_0} \Bigl(1 + O(p^{-1/2+\varepsilon})\Bigr) \\
      &= \frac{6}{\pi^2} \Bigl(\int_{\alpha p}^{\beta p} e^{-p/t} \frac{dt}{t^2}\Bigr)
           \Bigl(1 + O(p^{-1/2+\varepsilon})\Bigr) \\
      &= \frac{6}{\pi^2 p} \Bigl(\int_{1/\beta}^{1/\alpha} e^{-u}\,du\Bigr)
           \Bigl(1 + O(p^{-1/2+\varepsilon})\Bigr) \\
      &= \frac{6}{\pi^2 p} \bigl(e^{-1/\beta} - e^{-1/\alpha}\bigr)
            + O(p^{-3/2+\varepsilon}) \,.
  \end{align*}
  Letting $\alpha \to 0$ and $\beta \to \infty$, we obtain
  \[ \liminf_{p \to \infty}
       p \cdot \Prob\bigl(C(\F_p) \cap \langle \bar{P}_0 \rangle = \emptyset\bigr)
       \ge \frac{6}{\pi^2} \,.
  \]
\end{proof}

\begin{figure}[htb]
  \begin{center}
    \Gr{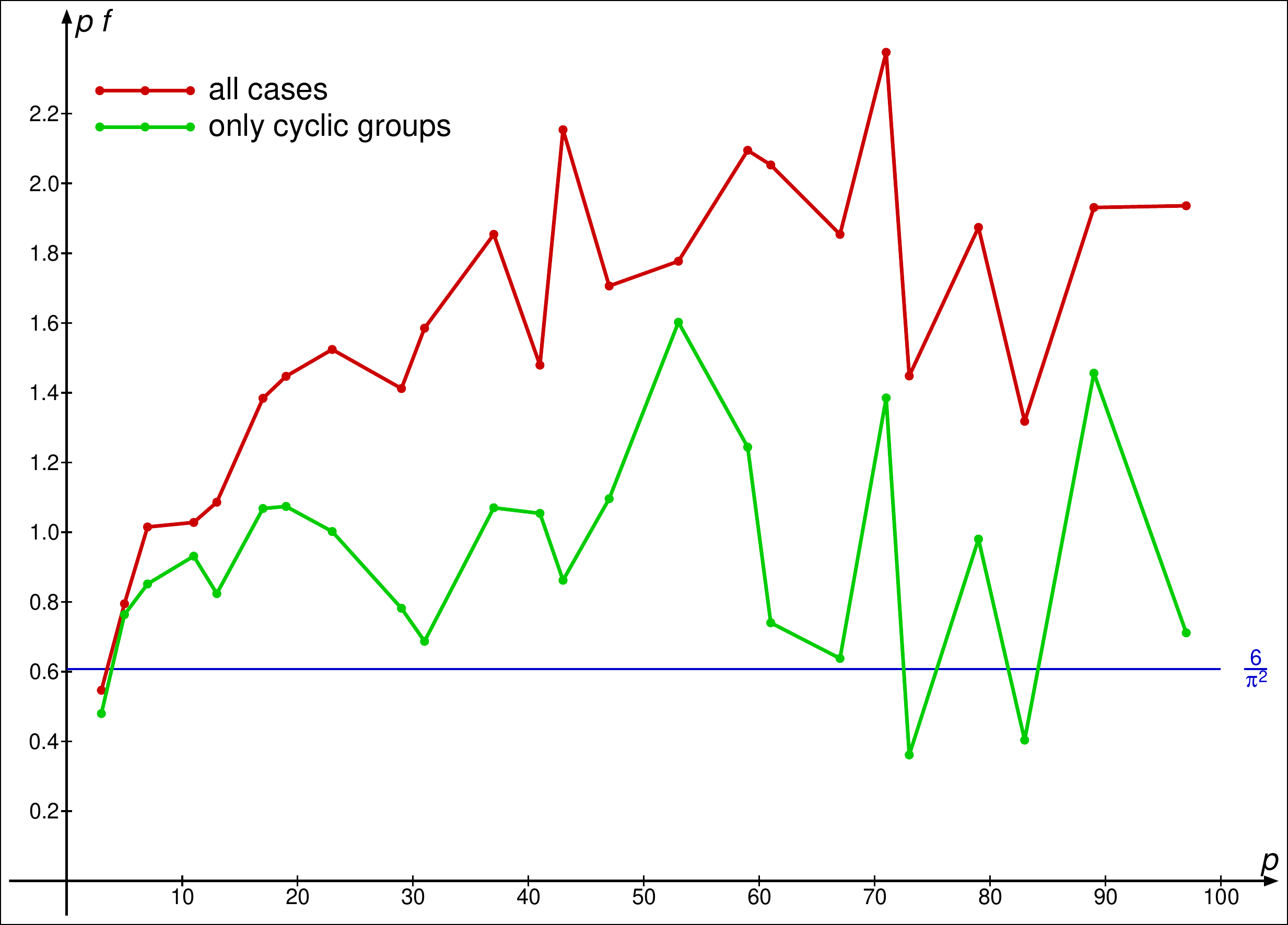}{0.8\textwidth}
    \caption{$p$ times success frequency at~$p$ versus~$p$} \label{fig:rk1-primes}
  \end{center}
\end{figure}

Since the cases with $d_0 \ll p$ and $d_0 \gg p$ are likely not to
contribute anything in the limit, we would expect that in the model considered,
we actually have
\[ \Prob\bigl(C(\F_p) \cap \langle \bar{P}_0 \rangle = \emptyset\bigr)
     \sim \frac{6}{\pi^2} \cdot \frac{1}{p} \quad \text{as $p \to \infty$.}
\]
Since $\sum_p p^{-1}$ diverges, we expect an infinite (but rather sparse) set
of primes~$p$ such that information mod~$p$ proves that there are no rational
points. This is consistent with the observations mentioned above.
Figure~\ref{fig:rk1-primes} shows $p$ times the fraction of curves in our
data set where reduction mod~$p$ proves the absence of rational points among
all curves with $r = 1$ and trivial torsion that have good reduction at~$p$,
as a function of~$2 < p < 100$. We see that (except for $p = 3$) this value
is considerably larger than $6/\pi^2$. The most likely explanation is that
this is an effect of the occurrence of non-cyclic groups among the~$J(\F_p)$.
This is confirmed by the data obtained from only looking at cases where
$J(\F_p)$ is cyclic (green in the figure).

In general, a similar heuristic approach should give a success probability
of the order of $p^{-r}$ when the rank is~$r$. This indicates that there is
a positive probability for success at some single prime, but this probability
is less than~$1$ and decreases to zero as $r$ increases. This is consistent
with the observations described below.

\medskip

\begin{figure}[htb]
  \begin{center}
    \Gr{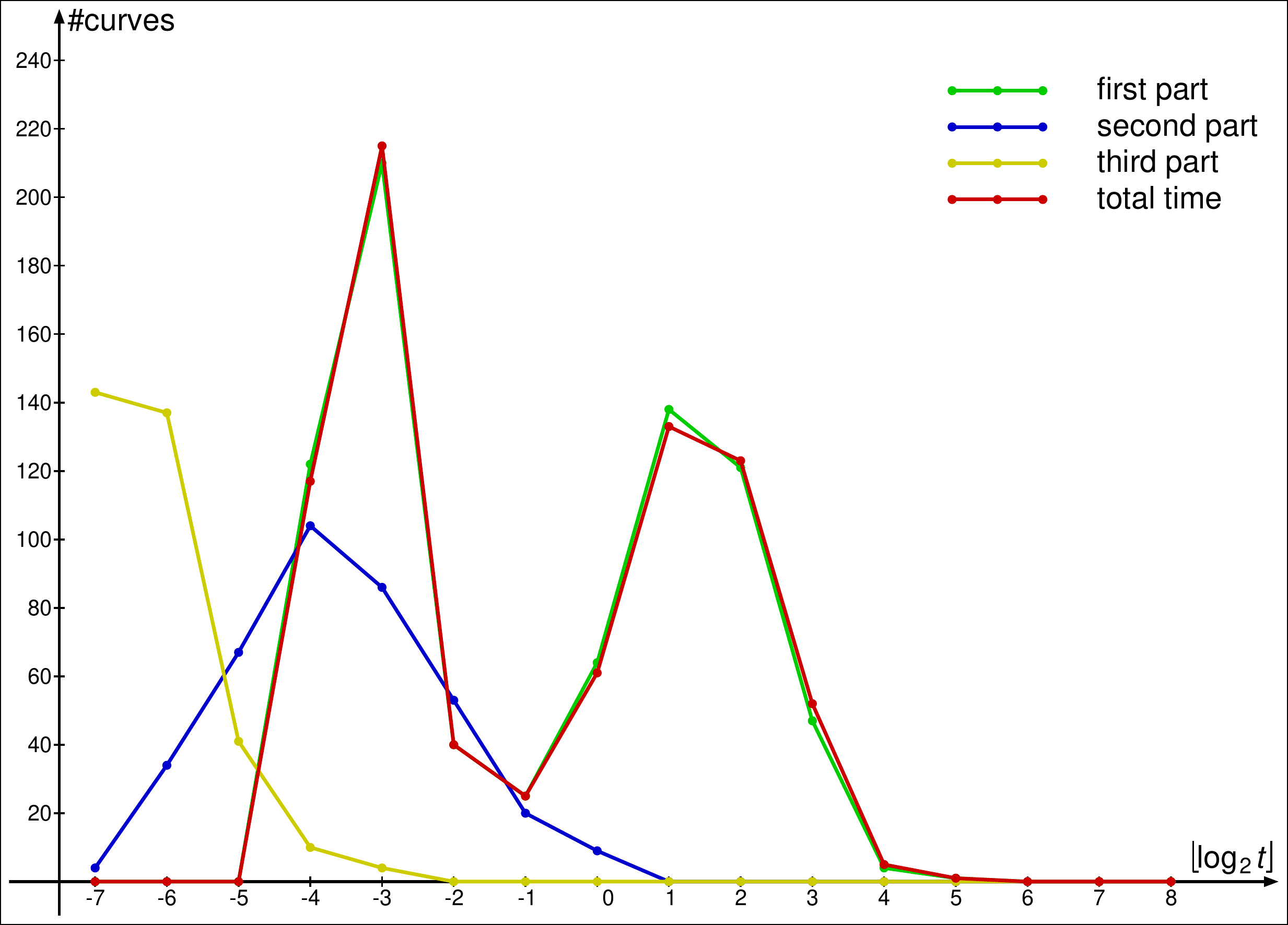}{0.8\textwidth}
    \caption{Running times for $r = 2$} \label{fig:rk2}
  \end{center}
\end{figure}

There are 772 curves with $r = 2$. For 394 among them, we obtain
a contradiction from one prime or prime power alone. 
The average computing time for these
curves was 0.24~seconds with a maximum of 6.4~seconds. For the remaining
curves, the average total computing time was 4.9~seconds, with a maximum
of 51.8~seconds. The distribution of the running times (overall and for
the various parts of the computation) is shown in Figure~\ref{fig:rk2}.
The two peaks essentially correspond to the two groups of curves. The largest size of
a set~$A(L)$ that occurred in the computation was~$236$, the average 
of this maximum size in each computation was~$6.1$.
Note that the inclusion of `bad' and `deep' information results in
a speed-up by roughly a factor two.

There are 152 curves with $r = 3$. For 14~curves, we still find
a contradiction from the local information at one prime alone. 
The average total time was 34.3~seconds, the maximum was about 5.6~minutes.
The first step took 28.1~seconds on average. For the curves where the second
and third steps were performed, the second step took 2.3~seconds and the
third step 4.6~seconds on average. The distribution of the running times
(overall and for the various steps) is shown in Figure~\ref{fig:rk3}.
The largest size of a set~$A(L)$ was~$251\,148$ (occurring for the curve
with the largest running time), the average was~$5049$.
For these curves, the computation is infeasible without using `bad' and `deep'
information, since otherwise the sets~$A(L)$ occurring in the last part
of the computation get much too large.

\begin{figure}[htb]
  \begin{center}
    \Gr{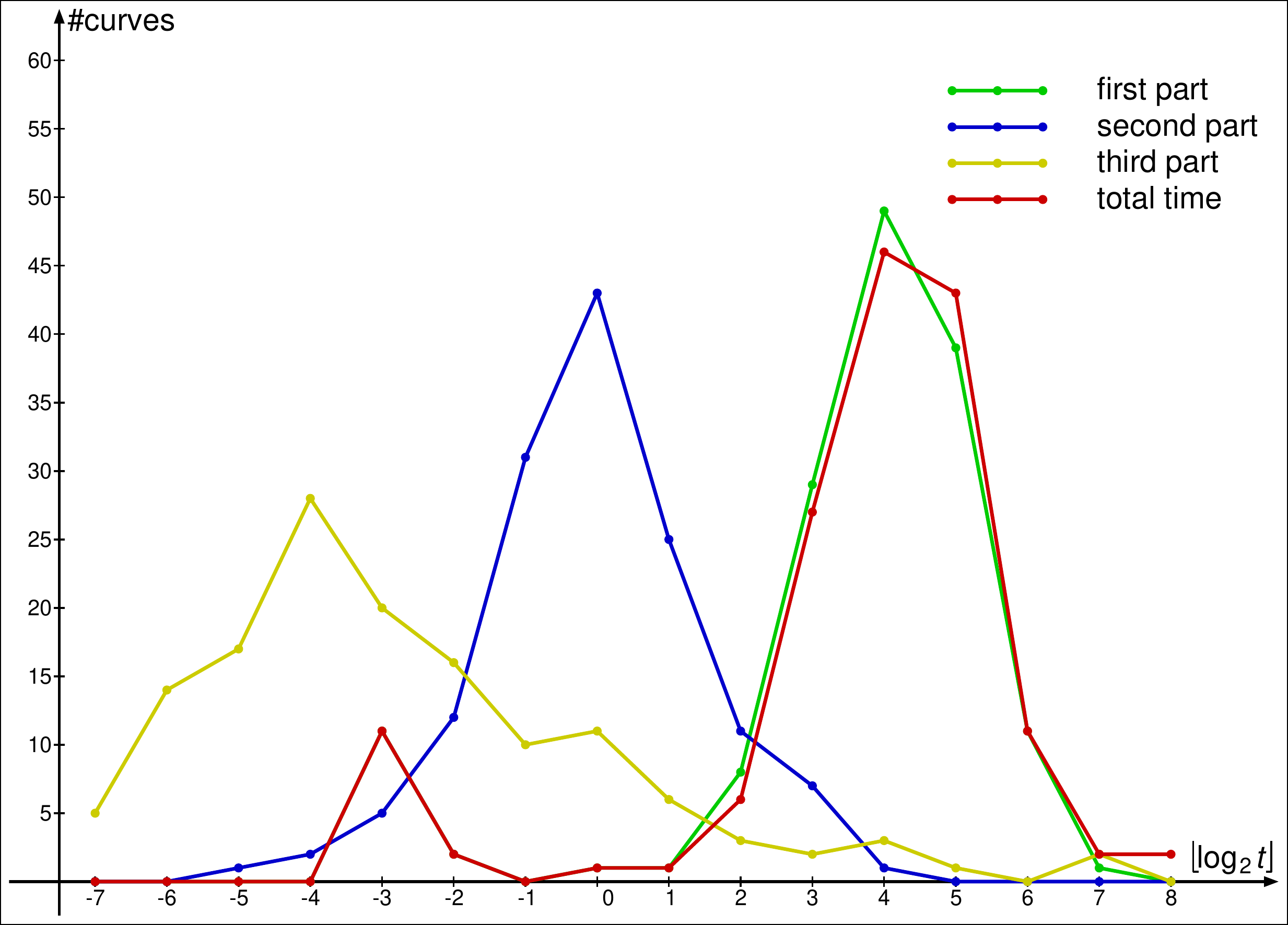}{0.8\textwidth}
    \caption{Running times for $r = 3$} \label{fig:rk3}
  \end{center}
\end{figure}

There are only two curves with $r = 4$. One of them is `hard' and the other
one is `easy'. For the `hard' curve, the computation takes about 26~minutes
with the standard settings (ca.~2~min for the first step, 10~seconds for
the second and the remaining 24~min for the sieving step). This is mostly
due to the large size of the sets~$A(L)$ (up to more than 2~million) occurring
in this computation. If we change the parameters so that deep information
mod~$p^n$ is used for all $p^n < 520$, then the computation takes less than
12~minutes (3~min, 10~sec, 8.5~min), and the largest set~$A(L)$ has size
only about $750\,000$. The `easy' curve is dealt with in 47~seconds
(44.5~sec, 2~sec, 0.5~sec) using the standard settings.

\medskip

From these data, we conclude that our current implementation works well
for curves with Jacobians of Mordell-Weil rank $r \le 3$. For larger rank,
there is so far only sparse evidence from examples, suggesting that
individual curves with $r$ as large as~$6$ are still within the range
of feasibility.
In any case, it is clear that average running times increase quickly
with~$r$.

Our timings also show that the first part of the computation (gathering
the local information) usually takes the lion's share of the total time.
Improvements in this part (and faster discrete logarithm computations
in particular) would result in a noticeable speedup of the procedure
as a whole.



\begin{thebibliography}{BMSST}
\frenchspacing

\bibitem[BW]{BW}
  {\sc A. Baker} and {\sc G. W\"ustholz},
  \emph{Logarithmic forms and Diophantine geometry}.
  New Mathematical Monographs 9
  (Cambridge University Press, Cambridge, 2007).

\bibitem[Br]{BruinCompMath}
  {\sc N. Bruin},
  {`The arithmetic of Prym varieties in genus~3',}
  \emph{Compositio Math.} {144} (2008) 317--338.

\bibitem[BE]{BruinElkies}
  {\sc N. Bruin} and {\sc Noam D. Elkies},
  {`Trinomials $ax^7+bx+c$ and $ax^8+bx+c$ with Galois groups of order 168 
   and 8$*$168',} 
   \emph{Algorithmic Number Theory, 5th International Symposium, ANTS-V,} 
   eds. Claus Fieker, David R. Kohel,
   Lecture Notes in Computer Science 2369 (Springer, 2002) 172--188.

\bibitem[BS1]{BruinStollExp}
  {\sc N. Bruin} and {\sc M. Stoll},
  {`Deciding existence of rational points on curves: an experiment',}
  \emph{Experiment. Math.} {17} (2008) 181--189.

\bibitem[BS2]{BruinStoll2D}
  {\sc N. Bruin} and {\sc M. Stoll},
  {`2-cover descent on hyperelliptic curves',}
  \emph{Math. Comp.} 78 (2009) 2347-2370.

\bibitem[BS3]{Software}
  {\sc N. Bruin} and {\sc M. Stoll},
  {\sf MAGMA} code for Mordell-Weil Sieve computations, \\
  {\tt http://www.mathe2.uni-bayreuth.de/stoll/magma/MWSieve-new.m}

\bibitem[BMSST]{IntegralPoints}
  {\sc Y. Bugeaud, M. Mignotte,  S. Siksek, M. Stoll} and {\sc Sz. Tengely},
  {`Integral points on hyperelliptic curves',}
  \emph{Algebra Number Theory} {2} (2008) 859--885.

\bibitem[Ca]{Cantor}
  {\sc David G. Cantor},
  {`Computing in the Jacobian of a hyperelliptic curve',}
  \emph{Math. Comp.} {48} (1987) 95--101.

\bibitem[CF]{CasselsFlynn}
  {\sc J.W.S. Cassels} and {\sc E.V. Flynn}, 
  \emph{Prolegomena to a middlebrow arithmetic of curves of genus~2}
  (Cambridge University Press, Cambridge, 1996).

\bibitem[Ch]{Chabauty}
  {\sc C. Chabauty},
  {`Sur les points rationnels des courbes alg\'ebriques de genre sup\'erieur 
   \`a l'unit\'e'} (French),
  \emph{C. R. Acad. Sci. Paris} {212} (1941) 882--885.

\bibitem[Co]{Coleman}
  {\sc R.F. Coleman},
  {`Effectve Chabauty',}
  \emph{Duke Math. J.} {52} (1985) 765--770.

\bibitem[Fl1]{FlynnCh}
  {\sc E.V. Flynn},
  {`A flexible method for applying Chabauty's theorem',}
  \emph{Compositio Math.} {105} (1997) 79--94.

\bibitem[Fl2]{Flynn}
  {\sc E.V. Flynn},
  {`The Hasse Principle and the Brauer-Manin obstruction for curves'},
  \emph{Manuscripta Math.} {115} (2004) 437--466.

\bibitem[Fl3]{FlynnFTP}
  {\sc E.V. Flynn}, FTP site with formulas relating to genus~2 curves, \\
  {\tt http://people.maths.ox.ac.uk/\verb+~+flynn/genus2/}

\bibitem[FS]{FlynnSmart}
  {\sc E.V. Flynn} and {\sc N.P. Smart},
  {`Canonical heights on the Jacobians of curves of genus $2$
   and the infinite descent'},
  \emph{Acta Arith.} {79} (1997) 333--352.

\bibitem[H]{Hartshorne}
  {\sc R. Hartshorne},
  \emph{Algebraic geometry.}
  Graduate Texts in Mathematics {52}
  (Springer, New York-Heidelberg, 1977).

\bibitem[M]{Magma} {\sf MAGMA} is described in
  {\sc W. Bosma, J. Cannon} and {\sc C. Playoust},
  {`The Magma algebra system I: The user language'},
  \emph{J. Symb. Comp.} {24} (1997) 235--265.
  (Also see the Magma home page at 
  {\tt http://www.maths.usyd.edu.au:8000/u/magma/}\,.)

\bibitem[McCP]{McCallumPoonen}
  {\sc W. McCallum} and {\sc B. Poonen},
  {`The method of Chabauty and Coleman'},
  Preprint (2007). Available at 
  {\tt http://www-math.mit.edu/$\sim$poonen/papers/chabauty.pdf}

\bibitem[MS]{MurtyScherk}
  {\sc V. Kumar Murty} and {\sc J. Scherk},
  `Effective versions of the Chebotarev density theorem for function fields',
  \emph{C. R. Acad. Sci. Paris S\'er. I Math.} 319 (1994) 523--528.

\bibitem[PH]{PohligHellman}
  {\sc G.C. Pohlig} and {\sc M.E. Hellman},
  {`An improved algorithm for computing
   logarithms over $\operatorname{GF}(p)$ and its cryptographic significance',}
  \emph{IEEE Trans. Information Theory} {IT-24} (1978) 106--110.

\bibitem[Po]{PoonenHeur}
  {\sc B. Poonen},
  {`Heuristics for the Brauer-Manin obstruction for curves',}
  \emph{Experiment. Math.} {15} (2006) 415--420.

\bibitem[PSS]{PSS}
  {\sc B. Poonen, E.F. Schaefer} and {\sc M. Stoll},
  {`Twists of $X(7)$ and primitive solutions to $x^2+y^3=z^7$',}
  \emph{Duke Math. J.} {137} (2007) 103--158.

\bibitem[Sh]{EMS}
  {\sc I.R. Shafarevich} (ed.),
  \emph{Algebraic geometry. I.}
  Encyclopaedia of Mathematical Sciences {23} (Springer-Verlag, Berlin, 1994).

\bibitem[Sc]{Scharaschkin}
  {\sc V. Scharaschkin},
  {`Local-global problems and the Brauer-Manin obstruction'},
  Ph.D. thesis, University of Michigan (1999).

\bibitem[St1]{StollH1}
  {\sc M. Stoll},
  {`On the height constant for curves of genus two',}
  \emph{Acta Arith.} {90} (1999) 183--201.
  
\bibitem[St2]{Stoll2DH}
  {\sc M. Stoll},
  {`Implementing 2-descent on Jacobians of hyperelliptic curves',}
  \emph{Acta Arith.} {98} (2001) 245--277.

\bibitem[St3]{StollH2}
  {\sc M. Stoll},
  {`On the height constant for curves of genus two, II',}
  \emph{Acta Arith.} {104} (2002) 165--182.

\bibitem[St4]{StollTw}
  {\sc M. Stoll},
  {`Independence of rational points on twists of a given curve',}
  \emph{Compositio Math.} {142} (2006) 1201--1214.

\bibitem[St5]{StollCov}
  {\sc M. Stoll},
  {`Finite descent obstructions and rational points on curves',}
  \emph{Algebra Number Theory} {1} (2007) 349--391.

\bibitem[St6]{StollOw07}
  {\sc M. Stoll},
  {`Applications of the Mordell-Weil sieve',}
  Oberwolfach Report 34/2007, 
  \emph{Oberwolfach Reports} {4} (2007) 1967--1970.

\bibitem[St7]{StollClay}
  {\sc M. Stoll},
  {`How to obtain global information from computations over finite fields',}
  \emph{Higher-dimensional geometry over finite fields},
  NATO Science for Peace and Security Series: 
  Information and Communication Security 16, 
  eds. D. Kaledin, Y. Tschinkel (IOS Press, 2008) 189--196.

\end{thebibliography}
\end{document}